\documentclass[10pt]{article}

\usepackage{ifthen}

\newboolean{draftcopy}
\makeatletter
\@ifclassloaded{amsart}%
{\setboolean{draftcopy}{false}}%
{\setboolean{draftcopy}{true}}%
\makeatother

\usepackage{geometry} 
\usepackage{ae} % or {zefonts}
\usepackage[T1]{fontenc}
\usepackage[cp1250]{inputenc}
\usepackage{amsmath}
\usepackage{amssymb, amsfonts,amscd,verbatim}
\usepackage{mathtools}
\usepackage{MnSymbol}
\usepackage{tikz}
\usetikzlibrary{positioning, arrows.meta}

\usepackage{relsize} % for mathlarger command

\usepackage{enumitem} % allows resuming enumerate counting

\usepackage{rotating}

\ifdraftcopy
\usepackage{tocloft} % todo table of contents
\usepackage{marvosym} % includes Coffeecup symbol
\usepackage[amsmath, thmmarks, thref]{ntheorem}
\fi
\usepackage{color}
\usepackage{soul}

\usepackage{tikz-cd}
\tikzset{>=latex}

\makeatletter
\newcommand\RedeclareMathOperator{%
  \@ifstar{\def\rmo@s{m}\rmo@redeclare}{\def\rmo@s{o}\rmo@redeclare}%
}
\newcommand\rmo@redeclare[2]{%
  \begingroup \escapechar\m@ne\xdef\@gtempa{{\string#1}}\endgroup
  \expandafter\@ifundefined\@gtempa
     {\@latex@error{\noexpand#1undefined}\@ehc}%
     \relax
  \expandafter\rmo@declmathop\rmo@s{#1}{#2}}
\newcommand\rmo@declmathop[3]{%
  \DeclareRobustCommand{#2}{\qopname\newmcodes@#1{#3}}%
}
\@onlypreamble\RedeclareMathOperator
\makeatother
\usepackage[normalem]{ulem}
\usepackage{hyperref}
\usepackage{indentfirst}
\usepackage{latexsym}
\input xy
\xyoption{all}

\usepackage[capitalise, noabbrev]{cleveref}
\usepackage{thmtools}
\usepackage{thm-restate}

\crefname{enumi}{}{}

\ifdraftcopy%
\newlistof{todoeasy}{easy}{Easy}
\newlistof{todomed}{med}{Medium}
\newlistof{todohard}{hard}{Hard}
\else{}%
\fi
\definecolor{lightGreen}{rgb}{.75, 1, .75}
\definecolor{lightYellow}{rgb}{1, 1, .65}
\definecolor{lightBlue}{rgb}{.75, .75, 1}

\usepackage{amsmath}    % need for subequations
\theoremstyle{plain}
\newtheorem{Pocz}{Poczatek}[section]
\newtheorem{Proposition}[Pocz]{Proposition}

\newtheorem{Theorem}[Pocz]{Theorem}
\newtheorem{Corollary}[Pocz]{Corollary}

\newtheorem{Fact}[Pocz]{Fact}

\newtheorem{Lemma}[Pocz]{Lemma}
\newtheorem{noproofLemma}[Pocz]{Lemma}

\newtheorem{Notation}[Pocz]{Notation}
\newtheorem{Definition/Theorem}[Pocz]{Definition/Theorem}

\newtheorem{Example}[Pocz]{Example}

\newtheorem{LetterTheorem}{Theorem}

\makeatletter
\@ifclassloaded{amsart}%
{\theoremstyle{definition}}%
{\theoremstyle{plain}}%
\makeatother
\newtheorem{Definition}[Pocz]{Definition}

\makeatletter
\@ifclassloaded{amsart}%
{\theoremstyle{remark}}%
{\theoremstyle{plain}}%
\makeatother
\newtheorem{Remark}[Pocz]{Remark}

\makeatletter
\@ifclassloaded{amsart}%
{\theoremstyle{plain}}%
{%
\theoremheaderfont{\sc} 
\theorembodyfont{\upshape}
\theoremstyle{nonumberplain}
\theoremseparator{}
\theoremsymbol{{\small \Coffeecup}}
\newtheorem{proof}{Proof:}
\theoremstyle{plain}
\theoremsymbol{{\small \Coffeecup}}
}%
\makeatother

\renewcommand{\epsilon}{\varepsilon}

\DeclareMathOperator{\supp}{supp}

\def\CC{{\mathbb C}}

\def\ZZ{{\mathbb Z}}
\def\NN{{\mathbb N}}
\def\TT{{\mathbb T}}

\def\UU{{\mathcal{U}}}
\def\VV{{\mathcal{V}}}

\def\HH{{H^{(0)}}}

\newcommand{\G}[1]{G^{(#1)}}
\newcommand{\Gn}{G^{(n)}}
\newcommand{\Gi}{G^{(i)}} %
\newcommand{\gpdG}[1]{\G{#1}}
\newcommand{\gpdH}[1]{H^{(#1)}}
\newcommand{\gpdL}[1]{L^{(#1)}} % used in the linking groupoids section

\newcommand{\fslash}{\ensuremath{/}}
\newcommand{\bslash}{\reflectbox{\fslash}}
\newcommand{\lorbit}[2]{\raisebox{-.2ex}{#1} \bslash {#2}} % group on the left, space acted on on the right
\newcommand{\rorbit}[2]{{#1} \fslash \raisebox{-.2ex}{#2}} % opposite from above

\newcommand{\sn}{\vskip10pt\noindent}

\RedeclareMathOperator*{\st}{st} % original \st is strikethrough from soul package
\numberwithin{equation}{section}

\newcommand{\coverU}{\mathcal{U}}
\newcommand{\coverV}{\mathcal{V}}
\newcommand{\coverW}{\mathcal{W}}

\newcommand{\restrict}[2]{{\left.#1\right|_{#2}}}
\newcommand{\cl}[1]{\overline{#1}} % use for closure of a set
\makeatletter
\@ifclassloaded{amsart}%
{% amsart will arrange the author, date and keywords information in its own style
\author{Kyle ~ Austin}
\thanks{The first author was funded by the Israel Science Foundation (grant No. 522/14).}
\address{Ben Gurion University of the Negev}
\email{ksaustin88@gmail.com}

\author{Magdalena ~ C. ~ Georgescu}
\thanks{The second author was supported by the ISF within the ISF-UGC joint research program framework (grant No. 1775/14), as well as by Israel Science Foundation grant no. 476/16.}
\address{Ben Gurion University of the Negev}
\email{georgesc@post.bgu.ac.il }

\title[Inverse Systems of Groupoids, with Applications to Groupoid $C^*$-algebras
]%
  {Inverse Systems of Groupoids, \\
with Applications to Groupoid $C^*$-algebras
}


}%
{% amsart not loaded, so format the usual title etc. manually
\title{Inverse Systems of Groupoids, \\
with Applications to Groupoid $C^*$-algebras}
\author{Kyle Austin\footnote{The first author was funded by the Israel Science Foundation (grant No. 522/14)}\hspace{.1cm} and  Magdalena C. Georgescu\footnote{The second author was supported by the ISF within the ISF-UGC joint research program framework (grant No. 1775/14), as well as by Israel Science Foundation grant no. 476/16.}}
}%
\makeatother

\begin{document}
%\fontsize{18}{20pt}\selectfont

\maketitle
%\begin{center}
%\today
%\end{center}

\begin{abstract}
We define what it means for a proper continuous morphism between groupoids to be Haar system preserving, and show that such a morphism induces (via pullback) a *-morphism between the corresponding convolution algebras. We proceed to provide a plethora of examples of Haar system preserving morphisms and discuss connections to noncommutative CW-complexes and interval algebras. We prove that an inverse system of groupoids with Haar system preserving bonding maps has a limit, and that we get a corresponding direct system of groupoid $C^*$-algebras.  An explicit construction of an inverse system of groupoids is used to approximate a $\sigma$-compact groupoid $G$ by second countable groupoids; if $G$ is equipped with a Haar system and 2-cocycle then so are the approximation groupoids, and the maps in the inverse system are Haar system preserving.  As an application of this construction, we show how to easily extend the Maximal Equivalence Theorem of Jean Renault to $\sigma$-compact groupoids.
\end{abstract}

\section{Introduction}

Many $C^*$-algebras can be modeled as groupoid $C^*$-algebras (see e.g. \cite{Renault}), which allows the use of the additional structural information to answer general questions in the theory of $C^*$-algebras. For example, J. L. Tu showed in \cite{Tu} that groupoids which satisfy the Haagerup property (e.g. amenable groupoids) have $C^*$-algebras which satisfy the UCT.  The collection of $C^*$-algebras to which this viewpoint applies is expanded further by considering twisted groupoid $C^*$-algebras.  In \cite{Renault-Cartan}, J. Renault shows that every Cartan pair must necessarily correspond to a Cartan pair $(C^*(G,\sigma),C_0(\G0))$ where $G$ is an \'etale groupoid with a twist $\sigma$.  One of the strengths of the groupoid approach to $C^*$-algebras comes from the ability to give geometric interpretations to $C^*$-algebra properties. For example, C. Laurent-Gengoux, J. L. Tu, and P. Xu conjecture in \cite{GTX} a geometric interpretation of classes in $K_0(C^*(G,\sigma))$ for Lie groupoids $G$ with a twist $\sigma$ as twisted vector bundles over the underlying groupoid $G$. In \cite{FG}, E. Gillaspy and C. Farsi extend these ideas to locally compact Hausdorff groupoids and give positive results for a class of \'etale groupoids.

Our main objective in this paper is to construct inverse systems of groupoids which induce direct systems of $C^*$-algebras and for which the groupoid $C^*$-algebra of the inverse limit groupoid is equal to the direct limit of the induced directed system of $C^*$-algebras. A key observation (see \cref{morphisms}) is that there are morphisms between groupoids with Haar systems such that the pullback map preserves the convolution product. We call such maps Haar measure preserving. A consequence of this observation is that inverse systems of groupoids with Haar measure preserving bonding maps induce direct system of $C^*$-algebras. In other words, we can construct a ``spectral'' picture at the level of groupoids for certain inductive limits of $C^*$-algebras. The following is our main existence theorem for inverse limits of groupoids, proven in \cref{proofoftheorema}:

\begin{restatable}{LetterTheorem}{limitexists}\label{mainthm.limitexists}
Let $\{G_\alpha,\sigma_\alpha, \{\mu_\alpha^y:y\in \G0_\alpha\},q^\alpha_\beta,A\}$ be an inverse system of groupoids with Haar systems and 2-cocycles and with bonding maps which are proper, continuous, surjective, Haar system preserving and cocycle preserving. The inverse limit groupoid $G= \varprojlim_{\alpha}G_\alpha$ exists and has a Haar system of measures $\{\mu^x:x\in \G0\}$ and 2-cocycle $\sigma$ such that $(q_\alpha)_*(\mu^x) = \mu_\alpha^{q_\alpha(x)}$ and $q_\alpha^*(\sigma) = \sigma_\alpha$.  Moreover, the pullback morphisms induce a direct system $\{C_c(G_\alpha,\sigma_\alpha), (q^\alpha_\beta)^*,A\}$ of convolution algebras that extends to a direct system of maximal completions (i.e. $C^*(G,\sigma) = \varinjlim_\alpha C^*(G_\alpha,\sigma_\alpha)$).
\end{restatable}

In a recent paper (see \cite{BussSims}),  A. Buss and A. Sims show  that  $C^*$-algebras that are not isomorphic to their opposite algebras cannot be groupoid $C^*$-algebras.  Nonetheless, there are a lot of examples of $C^*$-algebras that admit groupoid models, although uncovering the underlying groupoid structure of a $C^*$-algebra, if it exists, can be a non-trivial task. It is shown in \cite{EP} by R. Exel and E. Pardo  that every Kirchberg $C^*$-algebra is the groupoid $C^*$-algebra of an \'etale groupoid.  Another common technique for creating groupoids with prescribed groupoid $C^*$-algebras is via \'etale equivalence relations  on the Cantor set, see \cite{DPS} or \cite{P}. It is known that groupoid $C^*$-algebras of \'etale groupoids $G$ are simple if and only if $G$ is minimal and principal (see \cite{BCFS}). Hence, if one can find minimal and principal \'etale equivalence relations on the Cantor set with the correct $K$-theory, one can then appeal to the classification program of Elliot.  One possible application of \cref{mainthm.limitexists} is to use an inductive limit description of a specific $C^*$-algebra to construct an inverse system of groupoids whose dual is the given direct system.  However, \cref{mainthm.limitexists} is not general enough to handle some of the standard examples. In \cite{AM}, the first author and A. Mitra plan to adjust \cref{mainthm.limitexists} to construct groupoids whose groupoid $C^*$-algebras are equal to any UHF-algebra, infinite tensor powers of finite dimensional $C^*$-algebras, the Jiang-Su algebra, and the Razak-Jacelon algebra respectively. 

Inverse systems of groupoids of the type mentioned in \cref{mainthm.limitexists} are also constructed for our approximations of $\sigma$-compact groupoids.  This part of the paper arose out of an investigation of possible connections between geometric property (T) from coarse geometry and property (T) for topological groupoids.  Many groupoid results are stated for second countable groupoids; unfortunately, the groupoids arising from coarse geometry are not second countable, though they are $\sigma$-compact. The well known fact that $\sigma$-compact spaces are inverse limits of second countable spaces suggested to us a way of creating an ``easy'' bridge between results known in the second countable case to results about $\sigma$-compact groupoids. We accomplish this goal with the following theorem, proven in \cref{guts}: 

\begin{LetterTheorem}\label{mainthm.approximation}
Let $(G,\sigma)$ be a locally compact, Hausdorff and $\sigma$-compact groupoid with Haar system of measures $\{\mu^x:x\in G^0\}$ and $2$-cocycle $\sigma:\G2\to \TT$. One can obtain $(G,\sigma)$ as the inverse limit of an inverse system of  second countable, locally compact, and Hausdorff groupoids $\{(G_\alpha,\sigma_\alpha),q^\alpha_\beta,A\}$ in the category of locally compact groupoids with proper continuous groupoid morphisms. 

Moreover, there are induced Haar systems of measures and 2-cocycles on the $G_\alpha$'s, compatible with the bonding maps, such that the pullback morphisms induce a directed system of topological $*$-algebras $\{C_c(G_\alpha,\sigma_\alpha),(q^\alpha_\beta)^*,A\}$ such that $C_c(G,\sigma) = \varinjlim_\alpha C_c(G_\alpha,\sigma_\alpha)$.  This directed system can be extended to the maximal $C^*$-algebra completions.
\end{LetterTheorem}
In the above theorem, some of the properties of $G$ can be passed on to the approximation groupoids $G_\alpha$; for example, if $G$ is \'etale or transitive, the construction can be modified such that the same is true for all $G_\alpha$ (see  \cref{section.approximationproperties}).  The reason we did not state any results for the case of reduced groupoid $C^*$-algebras is because the reduced norms of the approximations are not necessarily compatible, not even with the reduced norm on $C_c(G)$.  This is an issue even in the case of groups, as can be seen from \cite{BKKO}.

The strategy to prove \cref{mainthm.approximation} is based on a classic uniform space theory construction: all uniform structures on a set $X$ are inverse limits of pseudo-metric uniform structures on $X$. The idea with this approximation technique is to start with a uniform cover $\mathcal{U}$ of $X$, construct a ``minimal'' uniform substructure on $X$ that contains $\mathcal{U}$, and throw away all other uniform covers. The basics of this construction are described in \cref{topapprox}.  In our case, we start with a uniform cover and show that one can construct ``minimal'' uniform structures on the underlying set of the groupoid $G$ such that the resulting metrizable space will be locally compact and Hausdorff and can be endowed with a groupoid structure.  We have to further sharpen our approximations if we want to also account for Haar systems and groupoid 2-cocycles. 

The interested reader might compare our construction for the proof of \cref{mainthm.approximation} to the groupoid approximations constructed for \'etale groupoids by Censor and Markiewicz in \cite{CM}; one notable difference is that our approximations are deconstructive whereas the approximations in their paper are constructive.  One novelty about our approximations is that our groupoid quotients (algebraic quotients) allow us to deform the object space and the space of arrows simultaneously.

As an application of our approximation construction in \cref{mainthm.approximation}, we show how to easily extend the maximal version of Renault's Equivalence Theorem. The notion of equivalence for groupoids was introduced  and was shown to be connected to  Morita equivalence of maximal groupoid $C^*$-algebras by J. Renault in \cite{Renault1} and soon thereafter by P. Muhly, J. Renault, and D. Williams in \cite{MRW}.  In \cref{renaultequivalencetheorem}, we prove the following theorem:

\begin{Theorem}[Equivalence Theorem]\label{equivalencetheorem}
Let $G$ and $H$ be $\sigma$-compact groupoids with Haar systems of measures. If $G$ and $H$ are equivalent then $C^*(G)$  is Morita equivalent to $C^*(H)$.
\end{Theorem}

%The purpose of including a proof for the above is to demonstrate how our approximation technique might be applied to extend known results from second countable to $\sigma$-compact groupoids.  We suspect, however, that the Equivalence Theorem holds in more generality and should follow from the more contemporary version of disintegration introduced by A. Buss, R. Holkar, and  R. Meyer in \cite{BHM}.

The purpose of including a proof for the above is to demonstrate how our approximation technique might be applied to extend known results from second countable to $\sigma$-compact groupoids. It came to our attention after we finished this work that the equivalence theorem is indeed true in full generality and was proven concurently and independently by A. Buss, R. Holkar, and  R. Meyer in \cite{BHM} using their generalization of the disintegration theorem of Renault.

\sn
\textbf{Acknowledgements}   The authors would like to extend their gratitude to Bill Chen, Adam Dor-On, Saak Gabrielyan,  Claire Anantharaman-Delaroche, Shirly Geffen, Jean Renault, and Dana Williams for their feedback and suggestions during the process of writing this paper.  We would like to offer special thanks to Joav Orovitz for his tremendous help throughout the project. This project would not exist without him. 

The first author would also like to thank  Michael Levin for all his helpful conversations and useful advice throughout the duration of this project.

\section{Preliminaries on Groupoids  and Their \texorpdfstring{$C^*$-Algebras}{C*-Algebras}} \label{section.groupoidsintro}

We use all the conventions and notations (except for second countability assumptions) as the survey paper  \cite{Buneci} by M. Buneci; we highlight the most important for ease of reference.  A \textbf{groupoid} is a small category in which every morphism is invertible. A groupoid $G$ is a \textbf{topological groupoid} if $\G1$ is equipped with a locally compact and Hausdorff topology for which the  inverse and partial multiplication functions are continuous.  We denote the partial multiplication map by $m$, and usually write $gh$ for $m(g, h)$; we use function composition order when composing arrows (i.e $gh$ defined if and only if $s(g) = t(h)$).  The object space $\G0 \subset \G1$ is given the subspace topology; then the source and target maps, $s(z) = z^{-1}z$ and $t(z) = zz^{-1}$ respectively, are continuous.  For $n\ge2$, the set $\Gn$ of composable $n$-tuples is given the subspace topology induced by the product topology on $G^n = G \times G \times \ldots \times G$ ($n$ times).

\begin{Definition}\label{morphism}
A \textbf{morphism of topological groupoids} is a continuous groupoid morphism (functor).
\end{Definition}

\begin{Definition}\label{opendef}
A topological groupoid is said to be \textbf{$\sigma$-compact} (resp. \textbf{paracompact}) if $\G1$ and hence\footnote{Recall that $\G0$ is a closed subset of $\G1$ and that paracompactness and $\sigma$-compactness are weakly hereditary properties.} $\G0$ have $\sigma$-compact (resp. paracompact) topologies.
\end{Definition}

\begin{Remark}
One can show that $\G2$ is a closed set in $G \times G$ (using the Hausdorff property of $\G0$ and continuity of $s$ and $t$). Therefore, if $\G1$ is $\sigma$-compact (resp. paracompact), then so is $\G2$.
\end{Remark}

\begin{Notation}
Let $G$ be a groupoid and let $x\in \G0$. We define $G^x$ to be the collection of arrows in $G$ that have target $x$, i.e. all $g\in G$ with $t(g) = x$. 
\end{Notation}

\begin{Definition}\label{Haarysituation}
Let $G$ be a topological groupoid. A \textbf{Haar system of measures on $G$} is a collection $\{\mu^x:x\in \G0\}$ of positive regular Radon measures on $G$ such that: 
\begin{enumerate}
\item $\mu^x$ is supported on $G^x$.
\item  For each $f\in C_c(G)$, the function $x\to \int_{G}f(g)d\mu^x(g)$ is continuous on $\G0$.
\item For all $g\in \G1$ and $f\in C_c(G)$, the following equality holds: 
$$\int_{G}f(h)d\mu^{t(g)}(h) = \int_{G}f(gh)d\mu^{s(g)}(h).$$
\end{enumerate}
\end{Definition}

Note that if $K$ is a compact set of arrows, the set of values its measure can take across the whole Haar system is bounded.  This simple observation will be needed in proving continuity results later, so we state it below as a lemma; the proof is straight-forward and hence omitted.

\begin{noproofLemma}\label{compactmaxmeasure} If $K \subset \G1$ is a compact set, then $\sup_{x\in \G0} \mu^x(K) < \infty$.
\end{noproofLemma}

\begin{Definition}\label{defcocycle}
 A 2-cocycle is a map $\sigma : \G2 \to \TT$ such that whenever $(g, h), (h, k) \in \G2$ we have
$$\sigma(g, h)\sigma(gh, k) = \sigma(g, hk)\sigma(h, k).$$
\end{Definition}

The map $\sigma(g, h) = 1$ for all $(g, h) \in \G2$ is always a 2-cocycle, called the \textbf{trivial cocycle}; the set of 2-cocycles on $G$ form a group under pointwise multiplication and pointwise inverse. 

\begin{Definition}
If $(G,\sigma_G)$ and  $(H,\sigma_H)$ are  locally compact groupoids with $2$-cocycles,  a morphism $q:(G,\sigma_G) \to (H, \sigma_H)$ is said to be \textbf{cocycle preserving} if it is a proper morphism of groupoids such that $\sigma_G(g,h) = \sigma_H(q(g),q(h))$. 
\end{Definition}

\begin{Definition}\label{approximate}
Let $G$ be a topological groupoid. An inverse system $\{G_\alpha, q^\alpha_\beta:G_\alpha \to G_\beta\}_{\alpha \in A}$ of topological groupoids with proper and surjective morphisms $q^\alpha_\beta$ and with directed indexing set $A$ is said to be an \textbf{inverse approximation of $G$} if 
\begin{enumerate}
\item $q^\alpha_\alpha = id_{G_\alpha}$ for all $\alpha$,
\item for each $\alpha \ge \beta \in A$ there exists $q^\alpha_\beta:G_\alpha \to G_\beta$ and, moreover, $q^\alpha_\beta \circ q^\beta_\gamma = q^\alpha_\gamma$ whenever $\alpha \ge \beta \ge \gamma$, and
\item $\varprojlim_\alpha G_\alpha = G$ in the category of topological groupoids (with proper continuous groupoid homomorphisms).  We denote the canonical projections from $G$ to the inverse system by $q_\alpha:G\to G_\alpha$. 
\end{enumerate}
\end{Definition}

If $(G, \sigma)$ is a topological groupoid with Haar system $\{\mu^x:x\in \G0\}$, we let $C_c(G,\sigma)$ denote the collection of compactly supported continuous complex valued  functions on $G$. With the following multiplication, adjoint operation, and norm $\|\cdot\|_I$, $C_c(G,\sigma)$ becomes a topological *-algebra. \label{ccdefinition}

$$ f*g(x) = \int_{G} f(xy)g(y^{-1})\sigma(xy,y^{-1})\,d\mu^{s(x)}(y) $$ 

$$ f^*(x) = \overline{f(x^{-1})\sigma(x,x^{-1})}$$

$$\|f\|_I = \max \left\{\sup_{x \in \G0} \int_{G} |f(g)|\,d\mu^x(g), \sup_{x \in \G0} \int_{G} |f(g^{-1})|\,d\mu^x(g)\right\}.$$

The\footnote{The designation "The" here is actually too strong, since in general Haar systems on groupoids are not unique. For our purposes, $G$ comes with a fixed Haar system, so we do not make a note of the Haar system in our notation. However, one should keep in mind that different Haar systems can give rise to different convolution algebras.  It is known that, in the second countable case, these algebras are all Morita equivalent (see section 4 of \cite{Buneci}).} \textbf{maximal twisted groupoid $C^*$-algebra of $G$}, denoted by $C^*(G,\sigma)$, is defined to be the completion of $C_c(G, \sigma)$ with the following norm
$$\|f\|_{max} = \sup_{\pi} \|\pi(f)\|,$$
where $\pi$ runs over all continuous (with respect to the norm $\|\cdot\|_I$) *-representations of $C_c(G, \sigma)$.  The maximal groupoid $C^*$-algebra of $G$ is obtained via the same construction when $\sigma$ is the trivial cocycle.

%\todoeasy{throw the following out if nothing can be done with the reduced case}
%Let $G$ be a locally compact groupoid with Haar systems of measures $\{\mu^x:x\in \G0\}$, 2-cocycle $\sigma:\G2\to \TT$ and let $\mu^0$ 
%be a Radon measure on $\G0$. Denote by $\mu$ to be the measure obtained by integrating the Haar system against $\mu^0$ and let $\mu^{1}$ be the image of the measure $\mu$ under the inversion map. Recall that the regular representation $Ind(\mu):C_c(G,\sigma)\to B(L^2(G,\mu^{-1}))$ is defined by $Ind(\mu)(f_1)(f_2) = f_1*f_2$ for all $f_1\in C_c(G,\sigma)$ and $f_2\in L^2(G,\mu^{-1})$. The reader can review the specifics of this representation in section 3.2 of \cite{Buneci}.

%Recall that the \textbf{reduced completion of $C_c(G,\sigma)$}, which we will denote by $C_r^*(G,\sigma)$, is the smallest completion of $C_c(G,\sigma)$ such that the regular representations $Ind(\delta_x)$ are continuous for all $x\in \G0$. Recall from Lemma 11 in \cite{SW} that, in the second countable case, this is equivalent to requiring $C_r^*(G,\sigma)$ to be the smallest completion of $C_c(G,\sigma)$ such that the regular representations $Ind(\mu)$ are continuous where $\mu$ ranges over all the finite Radon measures of $\G0$. %We will need to use this fact later on in our proof of an extension of the reduced equivalence theorem.  

\section{Pushing Haar Systems to Quotients and \texorpdfstring{\\}{} Induced Morphisms of \texorpdfstring{$C^*$}{C*}-algebras}\label{morphisms}

As an introduction to the ideas of this section, consider first the case when $G$ and $H$ are locally compact \emph{groups}.  Suppose $\mu_G$ is a Haar measure on $G$, and $\phi:G\to H$ is a proper and surjective topological group homomorphism. It is well-known that $\phi_* (\mu_G) := \mu_G \circ \phi^{-1}$ is a Haar measure on $H$; we will denote this measure by $\mu_H$.  The usual pullback of $\phi$ induces a $\CC$-module  homomorphism $\phi^*: C_c(H)\to C_c(G)$. We claim that it is a $*$-algebra morphism.  Let $f_1, f_2 \in C_c(H)$ and $y \in G$. We define $F_1, F_2 \in C_c(H)$ by $F_1(h) = f_1(\phi(y)h)$ and $F_2(h) = f_2(h^{-1})$ for $h \in H$.   We have

\begin{align*}
(\phi^*(f_1) \ast \phi^*(f_2))(y) &= \int_G f_1(\phi(yx))f_2(\phi(x^{-1}))\,d\mu_G(x) \\
&= \int_G (F_1 \cdot F_2)(\phi(x))\,d\mu_G(x) \\
&= \int_H (F_1 \cdot F_2)(z)\,d\mu_H(z) \\
&= \int_H f_1(\phi(y)z)f_2(z^{-1})\,d\mu_H(z) \\
&= (f_1 * f_2)(\phi(y)) = (\phi^*(f_1 \ast f_2))(y).
\end{align*}
An even easier computation shows that $\phi^*$ preserves the adjoint and hence is a *-morphism. 

The main theme of this paper is the approximation of topological groupoids by their images under proper surjective morphisms, allowing us to approximate groupoid $C^*$-algebras by subalgebras, namely by the groupoid $C^*$-algebras of the approximation groupoids. It is not clear that if $G$ is a topological groupoid with Haar system then such an image of $G$ (under a proper surjective morphism) will have a Haar system. Even if it does admit a Haar system, it does not seem obvious to the authors that the pullback map should induce a *-morphism of convolution algebras. It is the purpose of this section to establish a criterion for (proper) morphisms $G\to H$ of groupoids such that 

\begin{enumerate}
\item The pullback map is a *-morphism from $C_c(H)$ to $C_c(G)$ (endowed with the algebra operations and norms described on p. \pageref{ccdefinition}).
\item A Haar system on $G$ passes to a Haar system on $H$.
\end{enumerate}

We first establish a criterion for when morphisms of groupoids with Haar systems of measures induce *-morphisms of convolution algebras.

\begin{Definition}\label{haarpreserving}
Let $G$ and $H$ be locally compact groupoids with Haar systems of measures $\{\mu^x:x\in \G0\}$ and $\{\nu^y:y\in H^{(0)}\}$.  A groupoid morphism $q:G\to H$ is said to be \textbf{Haar system preserving} if $q$ is proper and satisfies either of the following two equivalent (see \cref{pushingforward}) conditions:

\begin{enumerate}
    \item  For all $z\in H^{(0)}$ and for all $x \in q^{-1}(z)$, we have that  $q_*\mu^x = \nu^z$.
    \item  For all $f\in C_c(H)$, for all $z\in H^{(0)}$ and for all $x \in q^{-1}(z)$ we have $\int_G (f \circ q) \, d\mu^x = \int_H f d\nu^z$.
\end{enumerate}
\end{Definition}

\cref{haarpreserving} gives a large class of  morphisms of groupoids that induce $*$-morphisms of groupoid $C^*$-algebras, but does not cover all possibilities. For example, the inclusion  $\ZZ \hookrightarrow \ZZ \times \ZZ/2\ZZ$ does induce a $*$-morphism of convolution algebras (the restriction morphism)  but is not Haar system preserving. In \cite{AM}, the first author along with A. Mitra add the required flexibility to cover this case by considering partial morphisms.

In analogy with the situation for groups, proper groupoid morphisms which are Haar system preserving induce *-morphisms of the convolution algebras:

\begin{Proposition}\label{inducedmorphism}
Let $q:(G,\sigma_G)\to (H,\sigma_H)$ be a cocycle and Haar system preserving morphism of locally compact Hausdorff groupoids with Haar systems $\{\mu^x:x\in \G0\}$ and $\{\nu^y:y\in H^{(0)}\}$, respectively. The pullback map $q^*:C_c(H,\sigma_H)\to C_c(G,\sigma_G)$ is a *-morphism of topological *-algebras.  If additionally $q$ is surjective, then $q^*$ is I-norm preserving. 
\end{Proposition}
\begin{proof}
%
%Let $f_1, f_2 \in C_c(H,\sigma_H)$ and consider
%
%\begin{align*} q^*(f_1)*q^*(f_2)(x) &= \int_{G}f_1(q(xy))f_2(q(y^{-1}))\sigma_G(xy,y^{-1})d\mu^{s(x)}(y) \\ %&= \int_{G} f_1(q(x)q(y))f_2(q(y)^{-1}))\sigma_H(q(x)q(y),q(y)^{-1})d\mu^{s(x)}(y) \\
%& = \int_{H}f_1(q(x)z)f_2(z^{-1}))\sigma_H(q(x)z,z^{-1})d\nu^{s(q(x))}(z) = q^*(f_1 * f_2)(x)
%\end{align*}
%
The calculations to show that $q^*$ respects convolution and adjoints are similar to those for group morphisms discussed at the beginning of this section, and are thus omitted.  We check that $q^*$ is continuous (and an isometry when $q$ is surjective).  Since $q$ is Haar system preserving, by definition, if $f\in C_c(H)$ and $z \in H$, then $\int_H |f| \,d\nu^z = \int_G |f \circ q| \, d\mu^x$ for all $x \in q^{-1}(z)$. From the definition of I-norm (see p. \pageref{ccdefinition}), it thus follows easily that the I-norm of $q^*f$ in $C_c(G, \sigma_G)$ is less than or equal to the I-norm of $f$ in $C_c(H, \sigma_H)$, with equality if $q$ is surjective.
\end{proof}

\begin{Fact}\label{pushingforward}
We give here a short proof of the fact that if $X$ and $Y$ are locally compact Hausdorff spaces and $f:X\to Y$ is a proper continuous function, then the pushforward of a regular Radon measure on $X$ is a regular Radon measure on $Y$.  It is easy to check inner regularity and local finiteness of the pushforward measure. To prove outer regularity, use the fact that proper maps to locally compact spaces are closed.  Then one can show that, for every $B\subset Y$ and every open neighborhood $U$ of $f^{-1}(B)$, the set $V = (f^{-1}(f(U^c)))^c$ is an open and saturated\,\footnote{Recall that a subset $A\subset X$ is saturated with respect to $f$ provided that $f^{-1}(f(A)) = A$.} neighborhood of $f^{-1}(B)$ such that $V \subset U$. Saturated open sets get mapped to open sets by closed maps, completing the proof.  

This observation about the pushforward of a regular Radon measure, along with the Riesz-Markov-Kakutani Theorem, gives us the equivalence of the two conditions presented in \cref{haarpreserving}.
\end{Fact}

Below, we discuss some conditions under which a Haar system of measures on a groupoid can be pushed to a quotient:
\begin{Proposition}\label{Haarquotient}
Let $G$ and $H$ be topological groupoids and let $q:G\to H$ be a proper surjective morphism.  Suppose moreover that $q$ is topologically a quotient map. If $G$ has a Haar system of measures $\{\mu^x:x\in \G0\}$ such that for all $f\in C_0(H)$, for all $z\in H^{(0)}$ and for all $x,y\in q^{-1}(z)$ we have 
$$\int_G (f \circ q) \, d\mu^x = \int_G (f \circ q) \, d\mu^y,$$
then $H$ admits a natural Haar system of measures $\{\nu^y:y\in H^{(0)}\}$ that makes $q$ Haar system preserving.
\end{Proposition}
\begin{proof}
For each $x \in H^{(0)}$ and each Borel subset $E\subset H$ we define $\nu^x(E) = q_*\mu^y(E)$ for any $y\in q^{-1}(x)$. By \cref{haarpreserving} and \cref{pushingforward}, the measure $\nu^x$ is well-defined and a regular Radon measure.  We check that $\{\nu^x : x \in \gpdH0\}$ is a Haar system; the fact that $q$ is then Haar system preserving follows immediately. 

Let $h\in H^{(1)}$, $g\in q^{-1}(h)$ and $f\in C_c(H)$. Notice
\begin{align*}
\int_H f(y) \, d\nu^{t(h)}(y) = \int_{H} f(y)\,d(q_* \mu^{t(g)})(y)  &= \int_{G} (q^*f)(y)\,d\mu^{t(g)}(y) \\
& =  \int_{G} (q^*f)(gy)\,d\mu^{s(g)}(y) \\
& = \int_{H} f(hy)\,d(q_*\mu^{s(g)})(y) = \int_H f(hy)\,d\nu^{s(h)}(y),
\end{align*}
and so the left invariance condition holds.

For the continuity of the Haar system, let $f\in C_c(H)$ and $\Omega \subset \CC$ be an open set. Let $V = \{x\in H^{(0)}: \int_{H} f(y) d\nu^x(y) \in \Omega\}$.  Since $q$ is proper, $f \circ q \in C_c(G)$; by the continuity of the Haar system on $G$, the set $W = \{z\in \G0: \int_G (f \circ q)(y) d\mu^z(y) \in \Omega \}$ is open in $\G0$.  The conditions on $q$ and the definitions of the measures $\nu^x$ ensure that $q^{-1}(V) = W$; since $q$ is a quotient map, $V$ must be open.  The continuity of the Haar system $\{\nu^x : x \in \gpdH0\}$ follows.
\end{proof} 

\subsection{Examples of Haar System Preserving Morphisms}

The concept of modeling $C^*$-algebras over topological spaces (i.e. preforming operations on a tensor factor of $C_0(X)$ for some locally compact Hausdorff space $X$) has had a tremendous impact to the theory of $C^*$-algebras, including in classification.  In \cite{ENST}, Elliot et al, as a stepping stone to proving that the decomposition rank of $\mathcal{Z}$-stable subhomogeneous $C^*$-algebras is at most $2$, prove that one can locally approximate any unital  subhomogenous $C^*$-algebra by noncommutative CW-complexes which, in the commutative case, have exactly the same topological dimension.
Jiang and Su in \cite{JS}, Razak in \cite{R}, Jacelon in \cite{J}, and Evans and Kishimoto in \cite{EK} successfully use interval algebras and the folding thereof to classify large classes of algebras and to create new examples of $C^*$-algebras. The following proposition is straightforward to prove and provides a powerful tool for modeling groupoids over topological spaces in an analogous way.

\begin{Proposition}\label{Exampleshaarmeasurepreserving}
Let $G$ be a topological groupoid with Haar system and let $f:X\to Y$ be a proper continuous function of locally compact Hausdorff spaces. The morphism $id_G\times f: G\times X\to G\times Y$ is Haar measure preserving (here, we take the Haar system to be $\mu^y \times \delta_x$ for $(x,y)\in \G0\times X$). 
\end{Proposition}

For the following examples, let $G$ be a topological groupoid with Haar system of measures.

\begin{enumerate}
    \item Let $X= \{0,1,2, \hdots n\}$ be an $n$-point set and notice that $G\times X \to G$ is Haar measure preserving by \cref{Exampleshaarmeasurepreserving} and the dual pullback morphism $f^*:C^*(G) \to \oplus_{i=1}^n C^*(G)$ is $f\to \oplus_{i=1}^n f$.
    \item More generally, let $X$ be a compact Hausdorff space and notice that the projection $\pi:G\times X \to G$ is Haar system preserving by \cref{Exampleshaarmeasurepreserving} and observe that the pullback morphism $\pi^*:C^*(G) \to C^*(G)\otimes C(X)$ is given by $f \to f\otimes id_{C(X)}$. 
    \item Let $X= [0,1]^n$ and let $Y= S^{n-1}$ and let $i:S^{n-1}\to [0,1]^n$ denote the inclusion. Notice that $G\times S^{n-1} \to G\times [0,1]^n$ defined by $id_G\times i$ is Haar system preserving. This example is extremely important for a future work of the first author in building groupoid models of noncommutative cell complexes.
\end{enumerate}

We conclude this section with examples of canonical morphisms of groupoids that have problematic pullbacks. In \cite{AM}, these constructions will be adjusted so that the pullbacks become unital maps between the groupoid $C^*$-algebras, allowing the authors to recover direct limit descriptions of standard $C^*$-algebras and construct new examples.

For each $n$, let $G_n$ denote the smallest groupoid whose object space consists of $n$ points and for which there exists an arrow between any two objects; i.e. $G_n$ is the product groupoid $\{1,2,3,\hdots n\}^2$. For example, the following are depictions of $G_2$ and $G_3$ respectively:

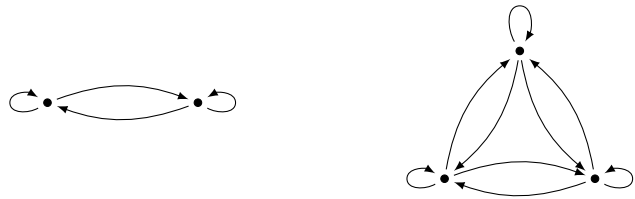
\begin{center}
\begin{minipage}[c]{.25\linewidth}
\begin{tikzpicture}
\node (a) at (0, 0) {};
\node (b) at (2, 0) {};

\filldraw (a) circle(0.05);
\filldraw (b) circle(0.05);

\draw[->] (a) to[out=-150, in=-90] (-0.5, 0) to[out=90, in=150] (a);
\draw[->] (b) to[out=-30, in=-90] (2.5, 0)  to[out=90, in=30] (b);
\draw[->] (a) to[out=20, in=160]  (b);
\draw[->] (b) to[out =-160, in=-20]  (a);
\end{tikzpicture}
\end{minipage}
\hspace{.5in}
\begin{minipage}[c]{.25\linewidth}
\begin{tikzpicture}
\node (a) at (0, 0) {};
\node (b) at (2, 0) {};
\node (c) at (1,1.7) {};

\filldraw (a) circle(0.05);
\filldraw (b) circle(0.05);
\filldraw (c) circle(0.05);

\draw[->] (a) to[out=-150, in=-90] (-0.5, 0) to[out=90, in=150] (a);
\draw[->] (b) to[out=-30, in=-90] (2.5, 0)  to[out=90, in=30] (b);
\draw[->] (c) to[out=120, in=180] (1, 2.3) to[out=0, in=60] (c);
\draw[->] (a) to[out=20, in=160]  (b);
\draw[->] (b) to[out = -160, in=-20]  (a);
\draw[->] (a) to[out=80, in=-140]  (c);
\draw[->] (c) to[out=-100, in=40]  (a);
\draw[->] (b) to[out=100, in=-40]  (c);
\draw[->] (c) to[out=-80, in=140]  (b);
\end{tikzpicture}

\end{minipage}

\end{center}

It is straightforward to check that the groupoid $C^*$-algebra $C^*(G_n)$ is the $n \times n$ matrix algebra $\mathbb{M}_n$. The most natural maps to consider here are the projections $G_n\times G_m \to G_n$.  Note that one needs to weight the Haar systems appropriately to make the projection map Haar measure preserving; even worse, the pullback morphism takes an $n\times n$ matrix $M$ to the matrix $\sum_{1\leq i,j\leq m} e_{i,j}\otimes M$ and hence the pullback map will not be unital. One can also see that any candidate for a bonding map from $G_n\times G_m \to G_n$ will essentially have the same problem, and that problem stems from the fact that the pullback of a function $f$ supported everywhere on $G_n$ will be supported everywhere on  $G_n\times G_m$, and hence cannot be of the form $M\to M\otimes id_{M_m}$.

\section{Inverse Systems: Proof of Theorem A}\label{proofoftheorema}

%\begin{Lemma}\label{gelfand}
%Let $\{X_\alpha,p^\alpha_\beta,A\}$ be an inverse system in the category of locally compact Hausdorff spaces with proper continuous maps. Furthermore, suppose that all the bonding maps are surjective. If $X=\varprojlim_\alpha X_\alpha$ then $\bigcup_{\alpha} C_0(X_\alpha)$ is dense in $C_0(X)$. In other words, $\varinjlim_\alpha C_0(X_\alpha) = C_0(X)$.
%\end{Lemma}
%\begin{proof}
%This is actually a direct consequence of the Gelfand duality. 
%\end{proof}

In this section we prove \cref{mainthm.limitexists}, restated here from the introduction for ease of reference:
\limitexists*
\begin{proof}
Let $G = \varprojlim_\alpha G_\alpha$ in the category of locally compact Hausdorff spaces (which exists and projects surjectively onto each of the factors in the inverse system). Let $q_\alpha:G\to G_\alpha$ denote the projections onto the pieces of the inverse system; by assumption, the $q_\alpha$'s are all proper and continuous. It is also easy to see that $G$, as a set, carries a groupoid structure such that the projections $q_\alpha$ are groupoid morphisms. We claim that the inversion and multiplication operations on $G$ are continuous.

To see that $m$ is continuous, let $U \subset G$ be open, and let $(x,y)\in m^{-1}(U)$. Because $G$ is a subspace of the product $\Pi_\alpha G_\alpha$, there exist $k \ge 1$, indices $\alpha_1, \ldots, \alpha_k$ and open sets $U_{\alpha_i} \subset G_{\alpha_i}$ such that $ xy\in  U_{\alpha_1}\times U_{\alpha_2} \times \hdots \times U_{\alpha_k} \times \Pi_{\alpha \notin \{\alpha_i: 1\leq i\leq k\}}G_\alpha \subset U$. As the multiplication\footnote{We use $m$ for the multiplication operation of any groupoid; the meaning should always be clear from context.} $m$ on $G_\alpha$ is continuous, for each $i$ there exists an open set $W_{\alpha_i} \subset \G2_{\alpha_i}$ containing $(q_{\alpha_i}(x),q_{\alpha_i}(y))$ such that $W_{\alpha_i}\subset m^{-1}(U_{\alpha_i})$. Notice that $(q_\alpha(x), q_\alpha(y))_\alpha \in W_{\alpha_1} \times W_{\alpha_2} \times \hdots \times W_{\alpha_k} \times \Pi_{\alpha \neq \alpha_i} \G2_\alpha \subset m^{-1}(U)$.  The same method of proof can be used to show that inversion is continuous. It follows that $G$ is a topological groupoid. 

We define $\sigma$ to be the inverse limit of the maps $\sigma_\alpha$; it is straightforward to see that it is continuous and that it satisfies the cocycle condition. 

Notice that, for each $x\in \G0$, $\{G_\alpha,\mu^{q_\alpha(x)}_\alpha,q^\beta_\alpha,A\}$ is an inverse limit of topological Borel measures spaces (see Definition 6 of \cite{Choksi}) that satisfies the maximal sequentiality condition (see Definition 4 of \cite{Choksi}) because our bonding maps are proper. By Theorem 2.1 of \cite{Choksi}, there exists a measure $\nu^x$ on $G$, defined on a sigma-subalgebra of the Borel sigma-algebra that contains a basis for $G$ and contains all compact subsets of $G$, such that the pushforward measure of $\nu^x$ along $q_\alpha$ is equal to $\mu_\alpha^{q_\alpha(x)}$ for each $\alpha$. It is easy to see that $\nu^x$ is positive as it is an inverse limit of positive measures (see the proof of Theorem 2.1 on page 325 of \cite{Choksi}). Using the Riesz-Markov-Kakutani Theorem, there exists a unique regular Radon measure $\mu^x$ such that integration against $\mu^x$ agrees with integration against $\nu^x$ for functions in $C_c(G)$. We claim that $\{\mu^x:x\in \G0\}$ is a Haar system for $G$.  Note first that the support of $\nu^x$, and hence also of $\mu^x$, must be contained in $G^x$ by construction.

To see that the system $\{\mu^x:x\in \G0\}$ is continuous, let $f\in C_c(G)$ and $\epsilon >0$. Choose $M \ge sup_{y\in \G0} \mu^y(supp(f))$ such that $M > 0$ ($M$ exists by \cref{compactmaxmeasure}).  Using a standard partition of unity argument, we can find $\alpha\in A$ and $h_\alpha \in C_c(G_\alpha)$ such that for $h = q_\alpha^*(h_\alpha)$ we have $|f - h| \leq \frac{\epsilon}{3M}$ and $supp(h) \subset supp(f)$. Let $x\in \G0$ and notice that $U_\alpha := \{y\in \G0_\alpha: \int h_\alpha \, d\mu^y_{\alpha} - \int h_\alpha \, d\mu^{q_\alpha(x)}_{\alpha} < \frac{\epsilon}{3}\}$ is an open subset in $\G0_\alpha$ and hence its pre-image $q_\alpha^{-1}(U_\alpha)$ is an open set in $\G0$ containing $x$. Notice also that if $y \in q^{-1}_\alpha(U_\alpha)$ then we have 
\begin{align*}
    \left|\int_G f d\mu^y - \int_G f d\mu^x\right| \leq \left|\int_G f d\mu^y-\int_G h d\mu^y\right| + \left|\int_G h d\mu^y-\int_G h d\mu^x\right| + \left|\int_G  h d\mu^x - \int_G f d\mu^x\right|.
\end{align*}
By the choice of $h$, we have that $\left|\int_G f d\mu^y-\int_G h d\mu^x\right| + \left|\int_G f d\mu^y-\int_G h d\mu^y\right| \leq 2\epsilon/3$  and, by properties of the pushforward measure, we have that $ \left|\int_G h d\mu^y-\int_G h d\mu^x\right| =  \left|\int_{G_\alpha} h_\alpha d\mu_\alpha^{q_\alpha(y)}-\int_{G_\alpha} h_\alpha d\mu_\alpha^{q_\alpha(x)}\right| \leq \epsilon/3$. Thus $y \in \{z\in \G0: \left|\int_G fd\mu^z - \int_G fd\mu^x\right|< \epsilon\}$. It therefore follows that the collection $\{\mu^x:x\in \G0\}$ satisfies the continuity assumption in \cref{Haarysituation}. 

The proof of left invariance follows essentially from the same kind of argument by using the fact that the measures $\mu^x_\alpha$ are left invariant.

$G$ is thus a topological groupoid with Haar system of measures and with 2-cocycle and the projection maps $q_\alpha$ are clearly Haar system preserving and cocycle preserving. It follows from \cref{inducedmorphism} that the pullbacks of the projection mappings induce $I$-norm embeddings from the directed system $\{C_c(G_\alpha,\sigma_\alpha),(p^\alpha_\beta)^*,A\}$ into $C_c(G,\sigma)$. The fact that $C^*(G)$ is the direct limit of the algebras $C^*(G_\alpha,\sigma_\alpha)$ follows from the fact that the union $\bigcup_{\alpha}C_c(G_\alpha,\sigma_\alpha)$ is dense in $C^*(G,\sigma)$. 
\end{proof}

\section{Inverse Approximation of Uniform Spaces}\label{topapprox}

%Let $\mathcal{C}$ be any category. By an \textbf{inverse system in $\mathcal{C}$}, we mean an collection of data $\{X_\alpha,p^\beta_\alpha,A\}$ where 
%\begin{enumerate}
%    \item $A$ is a directed set
%    \item $X_\alpha$ is an object of \mathcal{C}$ for every $\alpha$
%    \item if $\alpha > \beta$ then there exists a morphism $p^\alpha_\beta:X_\alpha\to X_\beta$ in $\mathcal{C}$ that satisfies
%    \begin{enumerate}
%    \item $p^\alpha_\alpha = id_{X_\alpha}$
%    \item $p^\alpha_\beta \circ p^\gamma_\alpha = p^\gamma_\beta$ for all $\gamma \ge \alpha \ge \beta$
%    \end{enumerate}
%\end{enumerate}

The goal here is to describe how uniform spaces (which we will be working with later) can be presented as inverse limits of metrizable spaces. Our approach will be to use covers, so we begin by reviewing some of the relevant terminology and notations.  We suggest\cite{Engelking} or \cite{Isbell}  as standard references on uniform spaces.  

Let $\mathcal{U}$ and $\mathcal{V}$ be covers of a set $X$.  $\mathcal{U}$ is said to \textbf{refine}  $\mathcal{V}$ (equivalently, $\mathcal{V}$ \textbf{coarsens} $\mathcal{U}$), written $\mathcal{U}\prec \mathcal{V}$, if each element of $\mathcal{U}$ is contained in some element of $\mathcal{V}$.

\begin{Definition}
Let $X$ be a set, $A\subset X$, and $\mathcal{U}$ be a cover of $X$. We define the \textbf{star of $A$ against $\mathcal{U}$}, denoted by $st(A,\mathcal{U})$, to be the set $\cup\{U\in \mathcal{U}: U\cap A\neq \emptyset\}$.
\end{Definition}

\begin{Remark}
The prototypical example of starring is given in metric spaces. If $X$ is a metric space and $\mathcal{U}$ is the cover of $X$ by $\epsilon$-balls then the star of a subset $A\subset X$ against $\mathcal{U}$ is contained in the $2\epsilon$-neighborhood of $A$ (they are equal if $X$ is a geodesic metric space).
\end{Remark}

\begin{Definition}
The \textbf{star of a cover $\UU$ against another cover $\VV$} is the cover 
$$st(\UU,\VV) := \left\{st(U,\mathcal{V}) :U\in \UU\right\}.$$ 
A cover $\mathcal{U}$ is said to \textbf{star refine} a cover $\mathcal{V}$, if $st(\UU,\UU)$  refines $\VV$. We will write $\mathcal{U} \leq \mathcal{V}$ if $\mathcal{U}$ star refines $\mathcal{V}.$ % in which case we say that $\UU$ is smaller than $\VV$  which will be written $\UU \leq \VV$ , or that $\VV$ is larger than $\UU$ (written $\VV \ge \UU$). 
\end{Definition}

\begin{Definition}
A \textbf{uniform space}   is a set $X$ equipped with a family $\Lambda = \{\coverU_\lambda\}$ of covers (called the \textbf{uniform covers} of $X$) such that
\begin{itemize}
    \item $\Lambda$ is closed under coarsening.
    \item For any $\coverU_1, \ldots, \coverU_n \in \Lambda$ there exists $\coverV \in \Lambda$ such that $\coverV \leq \coverU_j$ for all $j = 1, \ldots, n$.
\end{itemize}
A uniform space $X$ is called \textbf{Hausdorff} if, in addition, it satisfies:
\begin{itemize}
    \item For each $x,y\in X$ there exists $\mathcal{U} \in \Lambda$ such that there is no $U\in \mathcal{U}$ with $x,y\in U$.
\end{itemize}
\textbf{Unless stated otherwise, all the uniform spaces we consider in the following will be assumed to be Hausdorff.}
The only non-Hausdorff uniform structures we will work with are pseudo-metric (non-metric) uniform structures.

A function $f:X\to Y$ between uniform spaces is \textbf{uniformly continuous} if the pre-image of uniform covers are uniform covers.
\end{Definition}

As the name indicates, such a structure is used to abstract uniform properties of metric spaces, such as uniform continuity and uniform convergence.  

\begin{Definition}\label{normalsequencespace}
A \textbf{normal sequence of covers} of a set $X$  is a sequence $\{\mathcal{U}_n:n \ge 0\}$ of covers of $X$ such that $\mathcal{U}_{n+1} \leq \mathcal{U}_{n}$ for all $n\ge 0$. 
\end{Definition}

It is well known that a normal sequence of covers $\{\mathcal{U}_n:n \ge 0\}$ on $X$ defines a pseudo-metric on $X$. Here is an outline of the procedure: For elements $x,y \in X$, let $n(x,y)$ denote the maximum integer $k$ such that $x$ and  $y$ are both contained in an element of $\mathcal{U}_k$ and $\infty$ if no such maximum exists.  Let $\rho :X\times X \to [0,1)$ be defined by $\rho(x,y) = 2^{-n(x,y)}$, with the convention that $2^{-\infty} = 0$. We observe that $\rho$ itself is not necessarily  a pseudo-metric (because it may not satisfy the triangle inequality), but can be used to define a pseudo-metric $d$ via $d(x,y) = \inf \sum_{i=1}^n \rho(x_i,x_{i+1})$ where the infimum is taken over all chains $x = x_1, x_2, \hdots, x_n =y$ in $X$; $d$ satisfies the triangle inequality by definition. As shown in the proof of Theorem 14 on page 7 of \cite{Isbell} the cover by balls of radius 1 determined by the pseudo-metric $d$ refines $\mathcal{U}_1$. By an induction argument, the cover by balls of radius $\frac{1}{2^n}$ refines $\mathcal{U}_{n+1}$. The uniform structure generated by the resulting psuedo-metric is the structure whose uniform covers are precisely those which coarsen $\mathcal{U}_n$ for $n\ge 0$. We will write $\langle X, \{\mathcal{U}_n:n \ge 0\}\rangle$ to denote the resulting uniform structure.

\begin{Definition}\label{cofinalrefinement}
Let $\{\mathcal{U}_n\}_n$ and $\{\mathcal{V}_n\}_n$ be two normal sequences of covers. We will say that $\{\mathcal{V}_n\}_n$ \textbf{cofinally refines} $\{\mathcal{U}_n\}_n$ if for every $m\ge 0$ there exists $k(m)$ such that $\mathcal{V}_{k(m)}\leq \mathcal{U}_m$.
\end{Definition}

The following Lemma demonstrates the importance of cofinal refinement.

\begin{Lemma}\label{cofinallemma}
Let $\{\mathcal{U}_n\}$ and $
\{\mathcal{V}_n\}$ be normal sequences of covers. $\{\mathcal{U}_n\}$ cofinally refines $
\{\mathcal{V}_n\}$ if and only if the identity map $id_X:(X,\{\mathcal{U}_n\}) \to (X,\{\mathcal{V}_n\})$ is uniformly continuous.
\end{Lemma}

If $X$ is a uniform space then it is known that $X$ is the inverse limit of metrizable uniform spaces where the inverse limit is taken in the category of uniform spaces. Indeed, let $\{\mathcal{U}_n:n\ge 0\}$ be a normal sequence of uniform covers of $X$ and let $\langle X,\{\mathcal{U}_n\}\rangle$ denote the resulting pseudo-metric space. It is not difficult to show that the identity map $id_X:X\to \langle X,\{\mathcal{U}_n\}\rangle$ is uniformly continuous. One can show that if $\{\mathcal{U}_n:n\ge 0\}$ and $\{\mathcal{V}_n:n\ge 0\}$ are normal sequences of uniform covers then there exists a normal sequence of uniform covers that cofinally refines both. One thus has an inverse system indexed by the normal sequences of covers and ordered by cofinal refinement. It is not hard to show that the inverse limit of this system is precisely the given uniform space $X$. In order to show that $X$ is a limit of metrizable spaces, we modify the construction to quotient out in each pseudo-metric space by points which cannot be differentiated by the pseudo-metric.

A uniform structure on a set $X$ generates a topology in a way that is analogous to the way metrics define topologies. One defines it by declaring a set $A\subset X$ to be a neighborhood of a point $x\in X$ if there exists a uniform cover $\mathcal{U}$ of $X$ such that $st(\{x\},\mathcal{U}) \subset A$. We say that a collection of covers is a \textbf{base} for a uniform structure if it forms a uniform structure when one takes all coarsenings of covers in the given collection.

\begin{Definition}\label{paracompactness}(see Theorems 5.1.9 and 5.1.12 of \cite{Engelking}) 
A Hausdorff topological space $X$ is said to be \textbf{paracompact} if it satisfies any of the following equivalent conditions:
\begin{enumerate}
\item Every open cover of $X$ has a locally finite open refinement.
\item For every open cover $\mathcal{U}$, there exists a partition of unity whose carriers refine $\mathcal{U}$.
\item\label{starthatspace} Every open cover admits an open star refinement. 
\item \label{uniform} The collection of all open covers is a base for a uniform structure that generates the given topology.
\end{enumerate}
\end{Definition}

Condition \cref{uniform} says that paracompact spaces can be viewed as uniform spaces. In fact, they are the largest class of topological spaces which can be endowed with a uniform structure such that uniform concepts correspond directly to topological ones. \footnote{To address the possible objection that completely regular spaces should be that class, note that the collection of all open covers does not necessarily define a base for a uniform structure.  Therefore, one cannot guarantee that every continuous function between completely regular spaces is uniformly continuous with respect to any uniform structure that generates their topology.}

Recall that a Hausdorff topological space is said to be \textbf{Lindel\"of} if every open cover admits a countable open refinement. It is well known that every Lindel\"of space is paracompact and hence can be viewed as a uniform spaces.   It is easy to show that, for locally compact spaces, Lindel\"of is just a different guise for $\sigma$-compactness. Just to stress this (well-known) fact, which we will rely on in the future, we highlight it as a lemma:

\begin{Lemma}
Let $X$ be a locally compact Hausdorff space. $X$ is Lindel\"of if and only if $X$ is $\sigma$-compact.
\end{Lemma}

\begin{Lemma}\label{coverorfunction}
Let $X$ be a Lindel\"of  space and let $\{X_\alpha,p^\alpha_\beta,A\}$ be an inverse approximation of $X$ by metric spaces. For every $\alpha \in A$ denote by $q_\alpha$ the projection map $X \to X_\alpha$, and for each $n$ let $\coverV^\alpha_n$ be the cover of $X$ given by the pre-images under $q_\alpha$ of the cover of $X_\alpha$ by $\frac{1}{2^n}$-balls.  The following conditions are equivalent:
\begin{enumerate}[label=(\arabic*)]
    \item \label{item.coverrefinement} for every normal sequence $\{\mathcal{U}_n:n\ge 0\}$ of open covers for $X$, there exists $\alpha$ such that the normal sequence $\{\coverV^\alpha_n : n \ge 0\}$ cofinally refines $\{\mathcal{U}_n:n\ge 0\}$.
    \item \label{item.fnpullback} For every continuous function $f:X\to Y$ where $Y$ is a separable metric space, there exists $\alpha$ such that $f$  is a pullback of a uniformly continuous function $f_\alpha:X_\alpha\to Y$; i.e $f= f_\alpha \circ q_\alpha$.
\end{enumerate}
\end{Lemma}
\begin{proof}
Assume that condition \cref{item.coverrefinement} holds. Let  $f:X\to Y$ be a continuous function (not necessarily proper) where $Y$ is just some separable metric space. For each $n \ge 0$, let $\mathcal{W}_n$ be the pre-image under $f$ of the cover of $Y$ by balls of radius $2^{-n}$ . Notice that $\mathcal{W}_n$ is a normal sequence on $X$. By assumption, there exists $\alpha\in A$ such that $\{\mathcal{W}_n\}$ is cofinally refined by $\{\coverV_n^\alpha\}$. Notice that if $q_\alpha(x) = q_\alpha(y)$ then it must be the case that $f(x) = f(y)$. We define a continuous function $f_\alpha:X_\alpha \to Y$ by $f_\alpha(z) = f(x)$ for any $x \in q_\alpha^{-1}(z)$. It is straightforward to show that $f_\alpha$ is uniformly continuous on $X_\alpha$. Evidently, we have that $f = f_\alpha \circ q_\alpha$.

To show that \cref{item.fnpullback} implies \cref{item.coverrefinement}, note that we may assume that every normal sequence consists of countable covers, as the Lindel\"of property guarantees that these are cofinal (under the order of cofinal refinement) in the poset of normal sequences of covers. Every normal sequence $\{\mathcal{U}_n\}$ consisting of countable covers  induces a separable metric space.  By assumption, the projection $p$ of $X$ on the induced metric space $\overline{X}$ has the property that, for some $\alpha\in A$, $p$ is the pullback along $q_\alpha$ of a uniformly continuous function $p_\alpha:X_\alpha\to \overline{X}$. It follows that the pre-image under $p_\alpha$ of the covers of $\overline{X}$ by $\frac{1}{2^n}$ balls is a normal sequence on $X_\alpha$. Because the covers of $X_\alpha$ by balls of radius $\frac{1}{2^n}$ is cofinal (under cofinal refinement) in the collection of normal sequences of covers on $X_\alpha$, it follows that the normal sequence of covers of $X_\alpha$ by balls of radius $\frac{1}{2^n}$ cofinally refines the pre-image of the covers of $\overline{X}$ by balls of radius $\frac{1}{2^n}$. Hence, $\{\coverV^\alpha_n\}$ cofinally refines $\{\mathcal{U}_n\}$.
\end{proof}

\begin{Definition}[Definition/Notation]\label{exhaustion}
Let $X$ be a locally compact and $\sigma$-compact space. By an \textbf{exhaustion of $X$ by compact subsets} we will mean a nested collection $\{K_n:n\ge 0\}$ of compact neighborhoods such that $\bigcup_{i=1}^{\infty} int(K_n) = X$ ( where $int(A)$ denotes the interior of a subset $A\subset X$).%; we can add this extra requirement to our compact spaces by replacing each $K_n$ (if needed) by the union of sets in a finite cover consisting of neighbourhoods of points in $K_n$.
\end{Definition}

%\begin{Lemma}\label{lindelof=sigma}
%A locally compact space is Lindel\"of if and only if it is $\sigma$-compact.
%\end{Lemma}
%\begin{proof}
%Let $X$ be a locally compact space. If $X$ is Lindel\"of then one can form a countable cover of $X$ by pre-compact neighborhoods, so it is easily follows that $X$ is $\sigma$-compact. Conversely, if $X$ is $\sigma$-compact, one can take an exhaustion $\{K_i: i \ge 1\}$ of $X$ by compact sets.  Any open cover $\mathcal{U}$ can be finitely refined over each $K_i$, and the union of these finite refinements gives a countable open cover which refines $\mathcal{U}$. 
%\end{proof}

\begin{Proposition}\label{propermapping}
Let $X$ be a locally compact and $\sigma$-compact space and let $\{\mathcal{U}_n:n \ge 0\}$ be a normal sequence of locally finite open covers by relatively compact sets. Let $d(x, y)$ be the resulting pseudo-metric on $X$ (see the discussion following \cref{normalsequencespace}).  Let $\overline{X}$ denote the metric quotient of $(X,d)$ and let $q:(X,d)\to \overline{X}$ denote the quotient map. The following properties are satisfied:
\begin{enumerate}[label=(\arabic*)]
\item $id_X : X \to (X,d)$ is a proper map. Moreover, the composition of $id_X$ with the metrizable quotient map $q:(X,d) \to \overline{X}$ is proper. 
\item the quotient map $q:(X,d) \to \overline{X}$ is an open map; in fact, for every $x\in X$, we have $q(B(x,\epsilon)) = B(\overline{x},\epsilon)$.
\item \label{nesting} $st(x, \coverU_{n + 1}) \subset B(x, \frac{1}{2^n}) \subset st(x, \coverU_n)$, where $B(x, \frac{1}{2^n})$ is the ball around $x$ of radius $\frac{1}{2^n}$ (measured with respect to the pseudo-metric $d$). 
\end{enumerate}
\end{Proposition}
\begin{proof} The proofs of (2) and (3) follow directly from the definitions, except for the inclusion $B(x, \frac{1}{2^n}) \subset st(x, \coverU_n)$, which follows by applying induction to the argument given in the proof of Theorem 1 in \cite{Isbell}. 

To see (1), notice that the cover of $X$ by balls of radius $1$ refines $\mathcal{U}_0$ and therefore the balls of radius 1 in $X$ are relatively compact subsets of $X$. If $K_1$ is a compact subset of $(X, d)$ then it can be covered by finitely many balls of radius 1; hence $id_X^{-1}(K_1)$ is contained in a finite union of relatively compact subsets of $X$ and is therefore compact ($id_X^{-1}(K_1)$ is closed by the continuity of $id_X$).  A similar argument works to show that $q : (X, d) \to \overline{X}$ is proper, based on the fact that the image of the 1-balls in $(X, d)$ must also be relatively compact in $\overline{X}$, and the pre-image of a ball of radius $1$ of $\overline{X}$ is exactly a ball of radius $1$ in $(X, d)$.  The composition of two proper maps is proper, concluding the argument.
\end{proof}

\section{Approximations of \texorpdfstring{$\sigma$}{sigma}-compact groupoids: \texorpdfstring{\\}{}Proof of Theorem B}\label{guts}

We want to apply the ideas of \cref{topapprox} to groupoids in such a way that the resulting quotient object also has a groupoid structure.  To that end, we have to cover the objects and arrows separately by normal sequences; the interplay between the arrow covers and the object covers is a bit subtle, as can be seen from \cref{example.needsintersectionrefinement} and the technical conditions in \cref{Coverings}.

\begin{Example} \label{example.needsintersectionrefinement}   Consider the groupoid $G$ pictured below, consisting of 4 objects $\{x_1, x_2, x_3, x_4\}$ and all possible arrows $g_{ij}$ from $x_i$ to $x_j$ (in order to make the picture more readable, only a few arrow labels are shown):
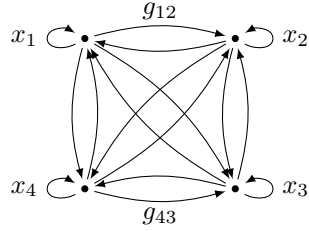
\begin{center}
\begin{tikzpicture}
\node (e1) at (0, 0){};
\node (e2) at (2, 0){};
\node (e3) at (2, -2){};
\node (e4) at (0, -2){};

% draw objects
\filldraw (e1) circle(0.04) node[anchor=north]{};
\filldraw (e2) circle(0.04) node[anchor=north]{};
\filldraw (e3) circle(0.04) node[anchor=north]{};
\filldraw (e4) circle(0.04) node[anchor=north]{};

% draw self-arrows
\draw[->] (e1) to[out=-150, in=-90] (-.5, 0) node[anchor=east]{$x_1$}  to[out=90, in=150] (e1);
\draw[->] (e2) to[out=-30, in=-90] (2.5, 0) node[anchor=west]{$x_2$} to[out=90, in=30] (e2);
\draw[->] (e3) to[out=-30, in=-90] (2.5, -2) node[anchor=west]{$x_3$} to[out=90, in=30] (e3);
\draw[->] (e4) to[out=-150, in=-90] (-.5, -2) node[anchor=east]{$x_4$} to[out=90, in=150] (e4);

% draw all arrows between different nodes
% horizontal arrows
\draw[->] (e1) to[out=15, in = 165]node[pos=.5, yshift=2mm]{$g_{12}$} (e2);
\draw[->] (e2) to[out=-165, in = -15] (e1);
\draw[->] (e4) to[out=-15, in = -165]node[pos=.5, yshift=-2mm]{$g_{43}$} (e3);
\draw[->] (e3) to[out=165, in = 15] (e4);
% vertical arrows
\draw[->] (e1) to[out = -105, in = 105](e4);
\draw[->] (e4) to[out=75, in = -75] (e1);
\draw[->] (e2) to[out = -105, in = 105](e3);
\draw[->] (e3) to[out=75, in = -75] (e2);
% diagonal arrows
\draw[->] (e1) to[out = -30, in = 120] (e3);
\draw[->] (e3) to[out = 150, in = -60] (e1);
\draw[->] (e2) to[out = -150, in = 60] (e4);
\draw[->] (e4) to[out = 30, in = -120] (e2);
\end{tikzpicture}
\end{center}
We endow $G$ with the discrete topology.  Consider the following open sets on $G$: $U_{00} = \{g_{11}, g_{33}\}$ and $U_{ij} = \{g_{ij}\}$ for all possible combinations of $i,j$ not equal to $1,1$ and $3,3$.  For $n \geq 1$, the collection $\coverU_{n}^1$ consisting of all these sets $\{U_{ij}\}$ (including $U_{00}$) is then a cover of the arrow space; $\coverU_{n}^0 = \{U_{00}, U_{22}, U_{44}\}$ is the restriction of this cover to the object space, and a cover of the object space in its own right (note that $x_i = g_{ii}$).

Since the $\coverU_n$'s form a normal sequence of covers, we can apply the construction of \cref{topapprox} to get a quotient space in which $x_1$ and $x_3$ are identified (since $U_{00}$ appears in all covers).  However, it should be apparent that if we want to place a groupoid structure on the resulting object, we have a problem with multiplication: $[g_{12}]$ has source $x_1$ and $[g_{43}]$ has target $x_3$, so since $x_1 \sim x_3$ we would expect the two to be composable; however, it is clear that it is not possible to define $[g_{12}] \cdot [g_{43}]$ in a way that is compatible with the structure on $G$.

The concept of intersection separating refinement which we introduce below (\cref{intersectionseparatingrefinement}) is needed in order to ensure that this kind of situation cannot occur and multiplication on the quotient is well-defined.  To check if the condition holds for our particular example, we construct a new collection of sets, one for each element $g_{ii}$, consisting of the intersection of all sets in $s(\coverU_n^1)$ which contain $g_{ii}$; this leads us to the collection of sets $U_{g_{ii}} = \{g_{ii}\}$.  Then $\coverV = \{U_{g_{ii}} : i = 1, \ldots,4\}$ is a cover of the object space, and one condition in the intersection separating requirement is that $\coverU_{n + 1}^0$ should refine $\coverV$; this is clearly not the case, because $U_{00}$ is contained in $\coverU_{n + 1}^0$ but is not a subset of any set in $\coverV$.  In other words, the normal sequence described here would be disqualified from consideration by condition \cref{intersectionseparatingcondition} of \cref{Coverings}.
\end{Example}

By Proposition I.4 from \cite{Westman}, if $G$ has a continuous Haar system of measures then the target and source maps are necessarily open.  Our proof of \cref{Coverings} does not need $G$ to have a Haar system, but does need the target and source maps to be open.

\begin{Notation}\label{groupoidcover}
If $\mathcal{U}$ is a cover of a set $X$ and $A\subset X$, we write $\mathcal{U}|_{A}$ to be the collection of elements of $\mathcal{U}$ that intersect $A$. 

Let $G$ be a topological groupoid. An \text{open cover} $\mathcal{U}$ of $G$ will consist of a pair $\{\mathcal{U}^0,\, \mathcal{U}^1\}$ of open covers of $\G0$ and $\G1$, respectively.
%, such that $\mathcal{U}^0 = \restrict{\coverU^1}{\gpdG0}$. 
\end{Notation}

\begin{Definition}\label{intersectionseparatingrefinement}
Let $\mathcal{U}$ be a \emph{finite} open cover of a locally compact Hausdorff space $X$. For each $x\in X$, let $U_x$ be the intersection of all elements of $\mathcal{U}$ that contain $x$. We define the  \textbf{intersection refinement of $\mathcal{U}$} to be the cover $\mathcal{U}' = \{U_x:x\in X\}$ (see \cite{CM} for a similar concept). Notice that $\coverU'$ must also be a finite cover.

Define an equivalence relation on $X$ by $x\sim y$ if $U_x=U_y$, and let $[[x]]$ denote the equivalence class of $x\in X$ under this equivalence relation. It is possible that the set $[[x]]$ is neither open nor closed in $X$. 

We say a cover $\mathcal{V}$ is an \textbf{intersection separating refinement of $\mathcal{U}$}, denoted $ \mathcal{V} \leq_{int} \mathcal{U}$,  if:
\begin{itemize}
\item $\coverV$ is a refinement of the intersection refinement $\coverU'$ described earlier, and 
\item whenever $\overline{[[x]]} \cap \overline{[[y]]} = \emptyset$ for some $x, y \in X$ then no open set in $\coverV$ contains elements from both sets $\overline{[[x]]}$ and $\overline{[[y]]}$ simultaneously.  
\end{itemize} 
Such a refinement always exists, because the equivalence relation defines a finite partition of $X$ and the normality of $X$ guarantees that we can always separate two closed disjoint sets by disjoint open sets.
\end{Definition}

\begin{Definition}\label{groupoidexhuastion}
Suppose $G$ is a $\sigma$-compact groupoid and the source and target maps are open.  Let $\{K_n\}$ be an exhaustion of $G$ by compact sets as in \cref{exhaustion}. Then $\{K_n':= K_n\cup t(K_n)\cup s(K_n)\}$ is also an exhaustion of $G$ by compact sets, and $K_n'|_{\G0}$ is an exhaustion of $\G0$ (both satisfying the requirement of \cref{exhaustion}).  We call an exhaustion of a $\sigma$-compact groupoid that has been obtained in this manner a \textbf{groupoid exhaustion}.  By construction, $K_n'$ has the property that $s(K_n'), t(K_n') \subset K_n'$.
\end{Definition}

The key step of our groupoid approximation argument is the following result, which gives a method for choosing a normal sequence of covers with very specific properties.  The consequence of these properties is explained in \cref{CoveringsExplanation}.

\begin{Proposition}\label{Coverings}
Let $G$ be a $\sigma$-compact groupoid with open source and target maps and let $\{K_n\}$ be a groupoid exhaustion of $G$ by compact sets (as in \cref{groupoidexhuastion}). Let $\{\mathcal{W}_n:n\ge0\}$ be a given normal sequence (see \cref{normalsequencespace}) of open covers of $G$ (see \cref{groupoidcover}). There exists  a normal sequence of countable and locally finite open covers  $\{\mathcal{U}_n\}_{n\ge 0}$ of $G$ such that for each $n \geq 1$ we have:
\begin{enumerate}[label=(\arabic*)]

\item \label{proper} $\mathcal{U}_n$ consists of  relatively compact open sets.
\item \label{refines} $\mathcal{U}_n\leq \mathcal{W}_n$.
\item \label{stinclusion} %$\restrict{\mathcal{U}^0_{n+1}}{K_n} \leq s(\restrict{\mathcal{U}^1_{n}}{K_n}),t(\restrict{\mathcal{U}^1_{n}}{K_n})$ and 
$ s(\mathcal{U}^1_{n}), t(\mathcal{U}^1_{n}) \leq \mathcal{U}^0_n$. 
\item \label{intersectionseparatingcondition} $\restrict{\mathcal{U}^0_{n+1}}{K_{n}} \leq_{ int}\; t(\mathcal{U}_n^1|_{K_n})\cup s(\mathcal{U}_n^1|_{K_n})$.
\item \label{subspacecondition} $\mathcal{U}_{n+1}^0 \leq \mathcal{U}_n^1|_{\G0}$.
\item \label{contmult} $m(\restrict{\mathcal{U}^1_{n + 1}}{K_n}, \restrict{\mathcal{U}^1_{n + 1}}{K_n}) \leq \mathcal{U}^1_{n}$.  
\item \label{inverse} $(\mathcal{U}^1_{n + 1})^{-1} \leq  \mathcal{U}^1_{n}$.
\end{enumerate}
If additionally $G$ is equipped with a Haar system of measures $\{\mu^x:x\in \G0\}$, then the sequence of covers can be chosen such that, for each $n \ge 0$, we have:
\begin{enumerate}[resume*]
\item \label{haaryfiber}  Fix $\{f^n_j:j \in J_n\}$ a finite partition of unity of $K_n$ whose carriers refine $\mathcal{U}^1_n$.  Let $(\lambda_j)_j \subset \CC$ be any sequence with $|\lambda_j| < n$.  For each open set $U\in \mathcal{U}^1_{n+1}$ and for each $x,y\in s(U)$ we have 
\begin{equation} \label{eqn.approximatespartition}
\left|\mathlarger{\int}_{G} \left(\sum_{j} \lambda_j f^n_j\right)\, d\mu^x - \mathlarger{\int}_{G}  \left(\sum_{j} \lambda_j f^n_j\right)\, d\mu^y\right| < \frac{1}{n}.
\end{equation}
\end{enumerate}
Moreover, if $\sigma : \G2 \to \TT$ is a 2-cocycle, we choose $\{\mathcal{V}_n\}_n$  a normal sequence of finite open covers of $\TT$ such that $\sup_{V\in \mathcal{V}_n}diam(V) \to 0$  as $n \to \infty$  and we can require that the sequence $\{\coverU_n\}$ satisfies: 
\begin{enumerate}[resume*]
\item \label{cocyclecond} $\sigma(\restrict{\coverU^1_n}{K_n}, \restrict{\coverU^1_n}{K_n}) \leq \coverV_n$.
\end{enumerate}
\end{Proposition}

\begin{proof}

Let $\mathcal{U}_0^0$ be any relatively compact open cover of $\G0$ which star refines $\mathcal{W}_0^0$, and $\coverU_0^1$ any relatively compact open cover of $\gpdG1$ which star refines $\coverW_0^1$, $s^{-1}(\coverU_0^0)$ and $t^{-1}(\coverU_0^0)$ simultaneously.  This ensures $\coverU_0$ satisfies \cref{proper}-\cref{stinclusion}; conditions \cref{intersectionseparatingcondition}-\cref{haaryfiber} deal with the interplay of consecutive covers in the normal sequence, so do not apply to $\coverU_0$.  If we need to satisfy \cref{cocyclecond} as well, then we could modify the construction of $\coverU^1_0$ as described later in this proof.  Since $\coverU_n$ will be chosen to refine $\coverU_0$ for all $n \geq 1$, the fact that $\coverU_0$ sets are relatively compact ensures that sets in each $\coverU_n$ are also relatively compact, hence  we do not need to consider condition \cref{proper} in the rest of the construction.

Now, assume $\mathcal{U}_n$ has been chosen for some  $n\ge 0$.  To construct $\coverU_{n + 1}^0$ we need to satisfy \cref{intersectionseparatingcondition} and \cref{subspacecondition}.  Since $s$ and $t$ are open maps, $s(\restrict{\coverU_n^1}{K_n}) \cup t(\restrict{\coverU_n^1}{K_n})$ is a finite open cover of $K_n \cap \gpdG0$, so we can apply the construction described in \cref{intersectionseparatingrefinement} to get an intersection separating refinement $\coverV$ of it.  Choose any star refinement of $\restrict{\coverU_n^1}{\gpdG0}$ and take $\coverU_{n + 1}^0$ to be a refinement of it such that $\restrict{\coverU_{n + 1}^0}{K_n}$ also refines $\coverV$. 

We construct separately, for $i = \ref{stinclusion}, \ref{contmult}, \ldots, \ref{cocyclecond}$, open covers $\mathcal{V}_i$ of $\gpdG1$ such that $\mathcal{V}_i$ satisfies condition $(i)$ (that is, $s(\coverV_3), t(\coverV_3) \leq \coverU_{n + 1}^0$, and so on for other values of $i$).  Then there exists a locally finite open cover (which we will take to be our $\mathcal{U}_{n + 1}^1$) which star-refines $\mathcal{U}_n^1$, $\coverW_n^1$ and all the $\mathcal{V}_i$'s simultaneously (the existence of such a cover is a consequence of paracompactness); $\coverU_{n + 1}$ would thus be a star-refinement of $\mathcal{U}_n$ satisfying all of \ref{refines}-\ref{cocyclecond}.  Covers satisfying conditions \ref{stinclusion} and \ref{inverse} are easily found using the continuity of the groupoid operations, so we only have to verify that we can find covers $\mathcal{V}_6$, $\mathcal{V}_8$, and $\coverV_9$.

%Using the continuity of $s$ and $t$, we choose a locally finite and relatively compact open cover $\mathcal{U}_0^1$ such that  $s(\restrict{\mathcal{U}^1_0}{K_0}),t(\restrict{\mathcal{U}^1_0}{K_0})\leq \mathcal{U}^0_0$.
%(note that first we have to construct a cover for $\G0$ and then the one for $\G1$ in order to satisfy condition \cref{stinclusion}). 

We now proceed to choose a cover $\mathcal{V}_6$ satisfying $m(\restrict{\mathcal{V}_6}{K_n},\restrict{\mathcal{V}_6}{K_n})\leq \mathcal{U}^1_n$.  Denote by $K$ the compact set $(K_n\times K_n) \cap \G2$.  Let $\coverW^{(2)}$ be a finite open refinement  of $\restrict{m^{-1}(\mathcal{U}^1_n)}{K}$ consisting of open sets of the form $(U\times V) \cap \G2$ where $U,V\subset \G1$ are open sets (such sets form a basis for the topology on $\G2$), and let $\coverW$ be the collection of sets $U \times V$ such that $(U \times V) \cap \G2$ appears in $\coverW^{(2)}$.

Let $N = \cup_{W \in \coverW} W$ (an open neighbourhood of $\G2$).  Extend $\coverW$ to a cover $\widetilde{\coverW}$ of $K_n \times K_n$ by adding in the following open sets: for each $(g, h) \in (K_n \times K_n) \setminus \G2$, pick a neighbourhood $U \times V$ of $(g, h)$ such that $U, V \subset G$ are open and $(U \times V) \cap \G2 = \emptyset$ (we can do this since $\G2$ is a closed set), and choose from these sets a finite subcover of the compact set $(K_n \times K_n) \setminus N$.

Let $\mathcal{F}$ be the collection of open sets $U$ for which there exists a $V$ such that either $U \times V$ or $V \times U$ is in the cover $\widetilde{\coverW}$.  Note that $\mathcal{F}$ is a finite cover of $K_n$.  For each $g\in K_n$, let 
$$U_g = \bigcap\limits_{U\in \mathcal{F},\, g\in U} U,$$
and $\coverV =\{U_g:g\in K_n\}$.  Since $\mathcal{F}$ is a finite collection, so is $\coverV$; moreover, $\coverV$ is an open cover of $K_n$ (in fact, $\coverV$ is the intersection refinement of $\mathcal{F}$, see \cref{intersectionseparatingrefinement}).
We claim that $(\coverV \times \coverV) \cap \G2$ refines $\coverW^{(2)}$.  Consider $U_g \times U_h$ for some $U_g, U_h \in \coverV$.  If $(g, h) \in U \times V$ for some $U \times V \in \coverW$ then by construction $U_g \times U_h \subset U \times V$, whence $(U_g \times U_h) \cap \G2 \subset (U \times V) \cap \G2 \in \coverW^{(2)}$.  On the other hand, if $(g, h) \in (K_n \times K_n) \setminus N$, then $(g, h) \in (U \times V)$ for some $U \times V \in \widetilde{\coverW} \setminus \coverW$, in which case $U_g \times U_h \subset U \times V \subset (K_n \times K_n) \setminus \G2$, and so $(U_g \times U_h) \cap \G2 = \emptyset$.
It follows that $(\coverV \times \coverV) \cap \G2 \prec \coverW^{(2)} \prec \restrict{m^{-1}(\mathcal{U}_n^1)}{K}$.  So $\coverV$ satisfies the multiplication condition on $K_n$, and to finish we just have to extend $\coverV$ to a cover of $G$.  We do this by adding in any cover of $G \setminus K_n$ by relatively compact sets, and we define $\mathcal{V}_6$ to be the resulting cover.

Note that the construction of $\coverV_9$ can be performed in a similar manner to that of $\coverV_6$, using the continuity of $\sigma : \G2 \to \TT$ and the chosen normal sequence of covers for $\TT$; as a consequence, we omit the details.

Suppose now that $G$ is equipped with a Haar system of measures and we also want to satisfy condition \cref{haaryfiber}. By the continuity of the Haar system, $x \mapsto \int_G f^n_j\,d\mu^x$ is continuous for each $j$, so using this and the triangle inequality we can find a cover $\coverV_8^0$ of $\G0$ such that for $x, y \in V \in \coverV_8^0$ and $|\lambda_j| < n$ we have
$$\left|\int_G \left(\sum\limits_j \lambda   _j f_j^n\right)\,d\mu^x - \int_G \left(\sum\limits_j \lambda_j f_j^n\right)\,d\mu^y\right| < \frac{1}{n}.$$
Choose $\coverV_8 = s^{-1}(\coverV_8^0)$. 

Use the covers $\coverV_i$ and $\coverW_n$ to construct a new cover $\coverU_{n + 1}$ as already described, concluding the proof.
\end{proof}
\begin{Definition}
   We call a normal sequence of covers which satisfies properties \cref{proper}-\ref{inverse} in \cref{Coverings} a \textbf{groupoid normal sequence} for $G$ (we include condition \ref{haaryfiber} as well if $G$ has a Haar system of measures, and condition \ref{cocyclecond} if there is an associated 2-cocycle).
\end{Definition}
\begin{Remark} \label{CoveringsExplanation}
In \cref{Coverings} we constructed a normal sequence which we constrained by a list of requirements.  We briefly explain the significance of each of these properties towards the construction of the approximation groupoid $G_\alpha$.
\begin{itemize} 
    \item The normal sequence $\{\coverW_n\}$ appearing in the hypotheses would usually come from a sequence of functions $\{f_n\}$, as explained in \cref{coverorfunction}.  This normal sequence is used in \cref{refines}, which says we can choose $\{\coverU_n\}$ such that the functions $f_n$ induce uniformly continuous functions $\overline{f}_n$ on the resulting approximation groupoid $G_\coverU$, in such a way that each $f_n$ is the pullback of $\overline{f}_n$.  This technique will be used to show $C_c(G)$ is the limit of continuous compactly supported functions on the approximations, see \cref{sexyapproximationtheorem}.
    \item \cref{proper} guarantees that the quotient map is a proper map (as shown in \cref{propermapping}).
    \item \cref{stinclusion} ensures that, in the quotient groupoid resulting from such a normal sequence, the target and source maps are well defined and continuous (as shown in \cref{quotientconstruction}).
    \item \cref{intersectionseparatingcondition} guarantees that our quotients are bona fide groupoids and not inverse semigroupoids (i.e. if $g$ and $h$ are arrows in the quotient with the target of $g$ equaling the source of $h$ we will have that $hg$ is an arrow, see Claim 3 in the proof of \cref{quotientconstruction}).
    \item \cref{subspacecondition} gives us that the object space topology induced by the covers $\{\mathcal{U}^0_n\}$ is the subspace topology defined by the inclusion $\gpdG0 \subseteq \gpdG1$ with the topology on $\gpdG1$ induced by the normal sequence $\{\mathcal{U}^1_n\}$. 
    \item \cref{contmult} makes sure that the partial multiplication on the quotient groupoid is well-defined and continuous (as shown in \cref{quotientconstruction}).
    \item \cref{inverse} gives us that the inversion map is a homeomorphism of our quotient (as shown in \cref{quotientconstruction}).
    \item \cref{haaryfiber} allows us to push the Haar system to a Haar system on the quotient groupoid (as shown in \cref{Haarsystemapproximation}). 
    \item \cref{cocyclecond} ensures we can push the cocycle $\sigma$ on $G$ to a cocycle $\sigma_\alpha$ on the quotient $G_\alpha$ (as shown in \cref{Haarsystemapproximation}).
    \item \cref{haaryfiber} and \cref{cocyclecond} guarantee that the induced map on convolution algebras $C_c(G_\alpha,\sigma_\alpha) \hookrightarrow C_c(G,\sigma)$ is a $*$-embedding (as shown in \cref{Haarsystemapproximation}).
\end{itemize}
\end{Remark}

% the following fact might not be true! 
%The following fact, which we decided to call a lemma, will be useful in the proof of the main theorem. The proof of the following lemma is straightforward and we leave it as an easy exercise.

%\begin{Lemma}\label{starynight}
%Let $\{\mathcal{U}_n:n\ge 0\}$ be a normal sequence of covers. For every $k >n+1$, each element of $\mathcal{U}_k$ is contained in every element of $\mathcal{U}_n$ that intersects it nontrivially. 
%\end{Lemma}

%Before proving the proposition, we will remind the reader about a classical technique for pushing forward open sets. Let $f:X\to Y$ be a closed map and $A\subset X$. The \textbf{saturation of $A$ with respect to $f$}, denoted by $sat(A)$, is defined to be the set $X\setminus f^{-1}(f(X\setminus A))$; note that $sat(A) \subset A$.  The closed assumption on $f$ easily implies that $sat(A)$ is open if $A$ is open.  Moreover, one can check that $f^{-1}(f(sat(A))) = sat(A)$; so if additionally $f$ is a quotient map, the definition of the quotient topology implies that $f(sat(A))$ is open whenever $A$ is open.

The following is well known and will be used throughout the proofs in this section.

\begin{Fact}\label{closed}
Let $X$ and $Y$ be locally compact Hausdorff spaces. If $f:X\to Y$ is a proper and continuous function then $f$ is closed.
\end{Fact}

%\begin{Corollary} \label{Gfnspaces.inclusion}
%Let $(G,\sigma)$ be a $\sigma$-compact open groupoid with cocycle $\sigma:\G2\to \TT$. Let $f^1:G\to Y_1$ and $f^0:\G0\to Y_2$ be  continuous mappings of to  second countable metric spaces $Y_i$. There exists a second countable and locally compact open groupoid $(G_f,\sigma_f)$ and a proper continuous, surjective, and cocycle preserving groupoid morphism $p:(G,\sigma)\to (G_f,\sigma_f)$ such that $f^1$ ($f^0$) is the pullback, via $p$, of a continuous function on $G_f$ ($\G0_f$).
%\end{Corollary}

\begin{Theorem} \label{quotientconstruction}
Suppose $G$ is a $\sigma$-compact groupoid with open source and target maps.  For any normal sequence $\{\mathcal{W}_n:n\ge 0\}$ of open covers (see \cref{normalsequencespace}) there exists a pseudo-metric structure $(G,d)$ on $G$ such that the induced metrizable quotient $G_\alpha$ of $(G,d)$ is a second countable, locally compact, Hausdorff groupoid.  Moreover, each morphism in the chain $\begin{tikzcd} G\arrow{r}{id_G} & (G,d) \arrow{r}{p_\alpha} & G_\alpha \end{tikzcd}$
is proper and continuous, and the pre-image of the cover of $G_\alpha$ by balls of radius $\frac{1}{2^n}$ refines $\mathcal{W}_n$.
\end{Theorem}

\begin{proof}

Let $\{K_n\}$ be a groupoid exhaustion by compact neighborhoods of $G$ (see \cref{groupoidexhuastion}), and use \cref{Coverings} to construct a groupoid normal sequence $\{\mathcal{U}_n\}_{n\ge 0}$ satisfying \cref{proper}-\cref{inverse} for $G$.

Let $d^i:\Gi \times \Gi$ be the psuedo-metric induced on $\Gi$ by the sequence of covers $\{\mathcal{U}^i_n:n\ge 0\}$ for $i =0,1$ (see \cref{topapprox} for details).  Recall that we defined an equivalence relation $x\sim^i y$ if and only if for every $n \ge 0$ there exists $U\in \mathcal{U}_n^i$ such that $x,y\in U$, and that the Hausdorff (and hence metrizable) quotient of $(G, d)$ is the quotient by this relation. Denote the equivalence class of $x\in G^{(i)}$ by $[x]^i$. Let $\Gi_\alpha$ be the Hausdorff quotient of the pseudo-metric space $(\Gi, d^i)$, and denote by $p^i_\alpha$ the quotient map $(\Gi, d^i) \to \Gi_\alpha$. Let $q_\alpha^i:\Gi\to \Gi_\alpha$ be equal to $p^i_\alpha \circ id_G $. By \cref{propermapping}, $q_\alpha^i$ is a proper and continuous map; as a consequence, $G_\alpha$ is locally compact.  We still need to check that $G_\alpha$ has an induced groupoid structure; furthermore, that $G_\alpha$ with the induced topology is a second countable topological groupoid. 

%Notice that conditions \cref{stinclusion} and \cref{intersectionseparatingcondition} guarantee that, when restricted to $K_n$, we have $\mathcal{U}^0_n\leq s(\mathcal{U}_{n-1}^1),t(\mathcal{U}^1_{n-1}) \leq \mathcal{U}^0_{n-1}$. Restrict all the covers to $\G0$ and this remains true. Notice that for $x\in \G0$ there is $N>0$ large enough so that $x \in K_N$, and hence $st(x,\mathcal{U}_N^0) \subset st(x,s(\mathcal{U}^1_{N-1}))  \subset st(x,\mathcal{U}_{N-1}^0)$. It follows that the topologies of $\G0$ induced by the pseudo-metric on $\G1$ and $\G0$ are equal.
%
%
%
\sn
\textbf{Claim: } Condition \cref{stinclusion} for a groupoid normal sequence (see \cref{Coverings}) guarantees that $g\sim_1 h$ implies both $s(g) \sim_0 s(h)$ and $t(g)\sim_0 t(h)$. 

%To see this, choose $m\ge 0$ large enough to ensure that $g,h\in int(K_m)$ and for which there exists $W\in \mathcal{U}_m^1$ with $\{g,h\} \subset W  \subset int(K_m)$.

Since $g\sim_1 h$, we have that, for every $n \in \NN$, there exists $U\in \mathcal{U}^1_n$ such that $g,h\in U$; then, as $s(U)$ is contained in an element $V$ of $\mathcal{U}_n^0$ by condition \cref{stinclusion}, we get $s(g), s(h) \in V \in \mathcal{U}_n^0$.  Hence $s(g)\sim_0 s(h)$, and similarly $t(g)\sim_0 t(h)$.  This shows that we can define source and target maps from $\G1_\alpha$ to $\G0_\alpha$, which we will also call $s$ and $t$ (as the meaning will be clear from context), as follows:

$$s([g]^1) = [s(g)]^0 \text{ and } t([g]^1) = [t(g)]^0.$$
\sn
\textbf{Claim:} $s,t : \G1_\alpha\to \G0_\alpha$ are continuous.

We show this for $s$; the proof that $t$ is continuous is similar. Consider $V$ a relatively compact  open neighborhood of $[x]^0 \in \G0_\alpha$.  Let $g \in \G1$ be any element of $s^{-1}(x)$; we will find a neighbourhood $V_1$ of $[g]^1$ in $\G1_\alpha$ such that $s(V_1) \subseteq V$. Since $(p_\alpha^0)^{-1}(V)$ is open in $(\gpdG0, d^0)$, by \cref{propermapping} there exists an $l \in \NN$ such that $W = st(x,\mathcal{U}_l^0)$ satisfies $W \subset (p^0_\alpha)^{-1}(V)$.  Let $U = st(g,\coverU_l^1)$.  Any $U' \in \coverU_l^1$ which contains $g$ necessarily has $x \in s(U')$; by assumption \cref{stinclusion} on the covers, there exists some $W' \in \coverU^0_l$ for which $s(U') \subset W'$, and $x \in W'$ implies $W' \subset W$.  Since $U$ is the union of all such sets $U'$ it follows that $s(U) \subset W$, where by construction $W$ was chosen such that $p_\alpha^0(W) \subset V$.

Applying \cref{propermapping} again, $U_1 := B(g, \frac{1}{2^l}) \subset U$.  Then $V_1 := p_\alpha^1(U_1)$ is an open set in $\G1_\alpha$ (since $p_\alpha^1$ is an open map) and we have
$$s(V_1) = s(p_\alpha^1(U_1)) = p_\alpha^0(s(U_1)) \subset p_\alpha^0(s(U)) \subset p_\alpha^0(W) \subset V,$$
allowing us to conclude that $s : \G1_\alpha \to \G0_\alpha$ is continuous.
\sn
\textbf{Claim}: The topology on $\gpdG0_\alpha$ induced by the covers $\{\coverU_n^0\}$ is the same as the subspace topology induced by the inclusion $\G0_\alpha \subset \G1_\alpha$. 

From condition \cref{subspacecondition} on the groupoid normal sequence (i.e. $\mathcal{U}_{n+1}^0 \leq \mathcal{U}_n^1|_{\G0}$ for $n \geq 1$) we get $st(x,\mathcal{U}_{n + 1}^0) \leq st(x,\mathcal{U}^1_{n})|_{\G0}$ for $n\ge 0$.  Hence the topology induced by the covers $\{\mathcal{U}_n^0\}$ is finer than the subspace topology on $\G0_\alpha$ induced by the covers $\{\mathcal{U}^1_n\}$. Conversely, use the fact that $s : \G1_\alpha \to \G0_\alpha$ is continuous (as proven above); then the identity map $\restrict{s}{\G0_\alpha}$ is still continuous, whence the subspace topology on $\G0_\alpha$ induced by the topology on $\G1_\alpha$ is finer than the topology on $\G0_\alpha$ induced by the covers $\{\mathcal{U}_n^0\}$.
\sn
We now turn our attention to the partial multiplication on $G_\alpha$.  Define $\G2_\alpha$ to be the collection of pairs $([g]^1,[h]^1)$ of elements of $\G1_\alpha$ with $[s(g)]^0 = [t(h)]^0$, and  $m_\alpha:\G2_\alpha \to \G1_{\alpha}$ by 
$$m_\alpha([g]^1,[h]^1) = [\{g'h' \in G: g'\in [g]^1, h'\in [h]^1 \text{ and } s(g') = t(h')\}].$$ 
The above will not define a multiplication operation on $\G2_\alpha$ if there exist $g, h \in G$ with $[s(g)]^0 = [t(h)]^0$ for which the set on the right hand side of the definition is empty. In the next claim we show this situation cannot occur. 
\sn
\textbf{Claim:} Condition \cref{intersectionseparatingcondition} on the groupoid normal sequence gives us that if $g,h\in G$ satisfy $[s(g)] = [t(h)]$ then there exist $g' \sim g$ and $h' \sim h$ such that $s(g') = t(h')$, i.e. $(g', h') \in \G2$. 

We will use the notation and terminology from \cref{intersectionseparatingrefinement}. Note that $\restrict{\mathcal{U}_n^1}{K_n}$ is a finite cover of $K_n$; denote the sets in the intersection refinement cover $(s(\restrict{\mathcal{U}_n^1}{K_n}) \cup t(\restrict{\mathcal{U}_n^1}{K_n}))'$ of $s(K_n)\cup t(K_n)$ by $U^n_x$; that is, $U^n_x = (\cap_U\, s(U)) \cap (\cap_V\, t(V))$, where the intersections are taken over all $U \in \coverU^1_n$ which intersect $K_n$ and $s^{-1}(x)$ and over all $V \in \coverU^1_n$ which intersect $K_n$ and $t^{-1}(x)$. 

Choose $N \ge 0$ large enough so that $st(g, \coverU_N^1), st(h, \coverU_N^1) \subset int(K_N)$.  If there exits $n \ge N$ such that  $\overline{U^n_{s(g)}} \cap \overline{U^n_{t(h)}} =\emptyset$ then $\overline{[[s(g)]]}\cap \overline{[[t(h)]]} = \emptyset$, and since $\restrict{\mathcal{U}_{n+1}^0}{K_n}\leq_{int}\; (s(\mathcal{U}_n^1|_{K_n}) \cup t(\mathcal{U}_n^1|_{K_n}))$, it must be the case that $\coverU^0_{n + 1}$ separates $s(g)$ and $t(h)$; that is, $[s(g)] \neq [t(h)]$, contradicting assumptions.

Hence we must have $\overline{U^n_{s(g)}} \cap \overline{U^n_{t(h)}} \neq \emptyset$ for every $n\ge N$.  Clearly $U^n_{s(g)} \subset s(st(g, \coverU^1_n))$; since each $U \in \coverU^1_n$ is relatively compact and $s$ is continuous, it is easy to see that $s(\cl{U}) = \cl{s(U)}$, so $\cl{s(st(g, \coverU^1_n))} = s(\,\cl{st(g, \coverU^1_n)}\,)$, whence $\cl{U^n_{s(g)}} \subset s(\,\cl{st(g, \coverU^1_n)}\,)$, and similarly with $s$ replaced by $t$ and $g$ replaced by $h$.  We can thus conclude that $s(\cl{st(g, \coverU^1_n)}) \cap t(\cl{st(h, \coverU^1_n)}) \not= \emptyset$ for all $n \geq N$. 

For $n \geq N$ let $T_n := s(\cl{st(g, \coverU^1_n)}) \cap t(\cl{st(h, \coverU^1_n)})$, a closed (in fact, compact) set contained in the compact set $K_N$.  Note that $\{T_n\}$ is a decreasing sequence of closed sets when ordered by inclusion, and hence has the finite intersection property.  It follows that there exists a $z \in \cap_{n \geq N} T_n$.

From the definition of $T_n$, for each $n$ there exists a $g_n \in \cl{st(g, \coverU^1_n)}$ such that $s(g_n) = z$ and an $h_n \in \cl{st(h, \coverU^1_n)}$ such that $t(h_n) = z$.  Since $\{g_n\}, \{h_n\} \subset K_N$ we can find limit points $g'$ of $\{g_n\}$ and $h'$ of $\{h_n\}$ respectively.  By the continuity of the source and target maps we must have $s(g') = t(h') = z$.  Moreover, from the construction it is clear that $d(g_n, g) \to 0$, so it follows that $g \sim g'$ and similarly $h \sim h'$, concluding the proof of the claim.

\sn
\textbf{Claim:} $m_\alpha([g]^1,[h]^1) = [gh]^1$ for $(g, h) \in \G2$.  

Fix any $(g, h) \in \G2$.  We check that $m_\alpha([g]^1,[h]^1) \subset [gh]^1$ (the other inclusion is obvious from the definition).  Suppose $g'\in [g]^1$ and $h'\in [h]^1$ are such that $g'h'$ is defined.  Choose $n_0 \geq 2$ large enough so $g, h, g', h' \in K_{n_0}$.  For every $n \geq n_0$, there exists $U, V \in \mathcal{U}_n$ such that $g,g' \in U$ and $h,h'\in V$ (since $g \sim g'$ and $h \sim h'$); it follows from condition \ref{contmult} on the groupoid normal sequence that $gh,g'h' \in m(U \cap K_n, V \cap K_n) \subset W$ for some $W$ an element of $\mathcal{U}_{n-1}$.  Since this holds for any $n \geq n_0$, we get $gh \sim g'h'$, so $g'h' \in [gh]^1$.  
\sn
\textbf{Claim:} $m_\alpha : \G2_\alpha \to \G1_\alpha$ is continuous.

Showing that $m_\alpha$ is continuous is similar to the proof that the source map is continuous, though, as in other claims involving the multiplication, one has to make sure to work in the framework of a specific $K_n$ for $n$ large enough.

\sn
\textbf{Claim:} There is a well-defined, continuous inverse operation on $\G1_\alpha$.

Define $([g]^1)^{-1} = [g^{-1}]^1$. From the definition of multiplication, it is immediate that $m([g]^1,([g]^1)^{-1}) = [id_{t(g)}]^1$ and that $m(([g]^1)^{-1},[g]^1) = [id_{s(g)}]^1$. Again, one can check that the inversion mapping on $G$ with the induced pseudo-metric is continuous to conclude that the inversion map is continuous on the metric quotient $G_\alpha$.

%By the same techniques applied throughout this proof, one can show that, by property \cref{inverse} of the groupoid normal sequence and by using the saturation technique, the inverse map $[x]^1 \to ([x]^1)^{-1}$ is well-defined and continuous on $G_\alpha$.
\sn
It follows that $G_\alpha$ is a topological groupoid. Each of the functions in the composition 
$$\begin{tikzcd} G\arrow{r}{id_G} & (G,d) \arrow{r}{p_\alpha} & G_\alpha \end{tikzcd}$$
is continuous. By Theorem 14 on page 7 in \cite{Isbell}, the pre-image of the cover of $G_\alpha$ by balls of radius $1$ refines $\coverW_0$, and using induction on the same argument we get that the pre-image of balls of radius $\frac{1}{2^{n}}$ refines $\mathcal{W}_n.$ 
%
%\todomed{qn: is it true that the induced source and target maps on $G_\alpha$ are open? Try to tackle this question}
\end{proof}

\begin{Theorem} \label{Haarsystemapproximation}
Suppose $G$ is a $\sigma$-compact groupoid  equipped with a Haar system of measures $\{\mu^x:x\in \G0\}$ and a 2-cocycle $\sigma : \G2 \to \TT$.  Given any normal sequence of open covers $\coverW_n$ of $G$, one can apply the construction of \cref{quotientconstruction} to get a second countable approximation $G_\alpha$ of $G$.  Additionally, moreover, the construction can be modified so $G_\alpha$ admits a Haar system of measures and a 2-cocycle $\sigma_\alpha$ for which the canonical pullback map $C_c(G_\alpha, \sigma_\alpha) \to C_c(G, \sigma)$ is a $*$-morphism of topological $*$-algebras (with respect to the I-norm or inductive limit topology).  
\end{Theorem}
\begin{proof} 
Let $\{K_n\}$ be a groupoid exhaustion by compact neighborhoods of $G$ (see \cref{groupoidexhuastion}), and use \cref{Coverings} to construct a groupoid normal sequence $\{\mathcal{U}_n\}$ satisfying conditions \ref{proper}-\ref{cocyclecond}.  The proof of \cref{quotientconstruction} shows how to construct the quotient groupoid $G_\alpha$ (and proves that the various groupoid operations are well defined); below, we show that we also get an induced Haar system on $G_\alpha$, and an induced 2-cocycle.  Denote by $(G, d)$ the groupoid $G$ with the topology generated by the pseudo-metric $d$ induced by $\{\coverU_n\}$.

We begin with the easier part, which is to construct $\sigma_\alpha$. Notice that $d(x,y) = 0$ for $(x, y) \in \G2$ implies, by condition \cref{cocyclecond} in \cref{Coverings}, that $\sigma(x) = \sigma(y)$. We may therefore define $\sigma_\alpha([x]) = \sigma(x)$. It is straightforward to check that this is continuous and satisfies the cocycle condition (see \cref{defcocycle}).

We move on to address the statements about Haar systems. Denote by $p_\alpha$ the metric quotient map $(G, d) \to G_\alpha$.  Consider $f \in C_c(G_\alpha)$, and let $\tilde{f}$ denote the composition $G \xrightarrow{id_G} (G,d) \xrightarrow{p_\alpha} G_\alpha \xrightarrow{f} \CC$. We will show that, given $\epsilon > 0$, there exists an $n$ such that if $x, y \in U \in \mathcal{U}_n^0$ then $|\int\tilde{f}\,d\mu^x - \int \tilde{f}\,d\mu^y|< \epsilon$; hence $x \sim y$ implies that $\int \tilde{f}\,d\mu^x = \int \tilde{f}\,d\mu^y$.

Since $p_\alpha \circ\, id_G$ is proper   (see \cref{quotientconstruction}), $\tilde{f} \in C_c(G)$.  Choose $M > \sup \mu^x(\supp(\tilde{f}))$ with $M > 0$, where the supremum is taken over all $x \in \gpdG0$ (the fact that the supremum exists follows from \cref{compactmaxmeasure}).  Let $\epsilon' =\frac{\epsilon}{3M}$.

Cover $\supp(\tilde{f})$ by open sets $\{V_1, \ldots, V_{k}\}$ where each $V_i$ is of the form  $\{y\in G: |\tilde{f}(x_i) - \tilde{f}(y)| < \epsilon'\}$ for some $x_i \in \G0$.  By a Lebesgue lemma argument and by \cref{propermapping}, there exists $m\ge 0$ such that, for each fixed $i$ and each $x \in V_i$ we have $U_x := st(x,\mathcal{U}^1_m) \subset V_i$.  We may assume that $m$ has also been chosen large enough so that $|\tilde{f}(x)|< m$ for all $x \in G$ and that $\supp(\tilde{f}) \subset int(K_m)$ (both of which we can do due to the fact that $\supp(\tilde{f})$ is compact), and also such that $\frac{1}{m} < \frac{\epsilon}{3}$.  

Consider the (finite) partition of unity $\{f^m_j\}_{j \in J_m}$ that we relied on in the construction of the groupoid normal sequence (see Property \ref{haaryfiber} of \cref{Coverings}).  We approximate $\tilde{f}$ by these functions as follows: for each $j$, if $\supp(f^m_j) \subset \supp(\tilde{f})$ then choose any $x_j \in \text{int}(\supp(f^m_j))$ and let $\lambda_j = \tilde{f}(x_j)$, else let $\lambda_j = 0$.  Note that $|\lambda_j| < m$ for each $j$.  Since the supports of the partition $\{f^m_j\}$ are subordinate to the cover $\{\mathcal{U}_m\}$, for any $x \in \supp(f^m_j)$ we have $|\tilde{f}(x) - \lambda_j| < \epsilon'$ (this is still the case when $\lambda_j = 0$ by choice, since if $\supp(f^m_j) \not\subset \supp(\tilde{f})$ then there exists a $y \in \supp(f^m_j)$ for which $\tilde{f}(y) = 0$, and so $|\tilde{f}(x)| < \epsilon'$ on $\supp(f^m_j)$).  We then have for any $x \in K_m$
\begin{align*}
|\tilde{f}(x) - \sum\limits_{j \in J_m} \lambda_j f^m_j(x)| &= |\tilde{f}(x) \sum\limits_{j \in J_m} f^m_j(x) - \sum\limits_{j \in J_m} \lambda_j f^m_j(x)| \\
&\leq \sum\limits_{j \in J_m} |\tilde{f}(x) - \lambda_j| \cdot f^m_j(x)
\end{align*}
For any fixed $j$, if $x \in \supp(f_j^m)$ then $|\tilde{f}(x) - \lambda_j| < \epsilon'$ by construction, and otherwise $f^m_j(x) = 0$.  In either case it is true that $|\tilde{f}(x) - \lambda_j|\cdot f^m_j(x) \leq \epsilon' f^m_j(x)$, and since the $f^m_j$'s add up to 1 at each $x \in K_m$ we get 
$$|\tilde{f}(x) - \sum\limits_{j \in J_m} \lambda_j f^m_j(x)| \leq \epsilon',$$ 
as desired.  Moreover, by construction, $\supp\left(\sum \lambda_j f^m_j\right) \subset \supp(\tilde{f}) \subset int(K_m)$.

It now easily follows that, for any $x \in \gpdG0$,
$$\left|\int_G \tilde{f}\,d\mu^x - \int_{G} \sum_{j} \lambda_j f_j^m\,d\mu^x\right| \leq \epsilon' \mu^x(\supp(\tilde{f})) <\epsilon'M = \frac{\epsilon}{3}.$$
By Property \cref{haaryfiber} of the groupoid normal system, for any $U \in \mathcal{U}^1_m$ we have $x, y \in s(U)$ implies (for $|\lambda_j| < m$, which we arranged for in the choice of $m$ and $\lambda_j$)
$$\left|\int_{G} \sum_{j} \lambda_j f_j^m\,d\mu^x - \int_{G} \sum_{j} \lambda_j f_j^m\,d\mu^y\right| < \frac{1}{m}<\frac{\epsilon}{3}.$$
Hence the triangle inequality gives us that for any $U \in \mathcal{U}_m$ and $x, y \in s(U)$ we have
$$\left|\int_{G} \tilde{f}\,d\mu^x - \int_{G} \tilde{f}\,d\mu^y\right| < \epsilon.$$
As a consequence of property \cref{intersectionseparatingcondition} of the groupoid normal sequence, $\restrict{\mathcal{U}^0_{m + 1}}{K_m}$ refines $\{s(U) \cap K_m : U \in \mathcal{U}^1_m\}$.  Hence for the given $\epsilon > 0$, if we take any $V \in \mathcal{U}^0_{m + 1}$ then $x, y \in (V \cap K_m) \subseteq s(U)$ implies 
\begin{equation}
\left|\int_{G} \tilde{f}\,d\mu^x - \int_{G} \tilde{f} \,d\mu^y\right| < \epsilon. \label{eqn.Haarsystemapproximation.continuity}    
\end{equation}
In particular it follows that for $x \sim y$ we must have $\int_G \tilde{f} \,d\mu^x = \int_G \tilde{f} \,d\mu^y$.  

For each $[x] \in \gpdG_\alpha^0$, we can thus define a positive linear functional on $\mathcal{C}_c\left(G_\alpha\right)$ by $f \mapsto \int_G \tilde{f} \,d\mu^x$, where $x$ is chosen to be any representative of $[x]$.  By the Riesz-Markov-Kakutani theorem, this defines a unique regular Radon measure $\mu^{[x]}$ on $G_\alpha$ (supported on $G_\alpha^{[x]}$); we claim that the collection $\{\mu^{[x]} : [x] \in G_\alpha^{0}\}$ defines a Haar system.

The fact that 
$$\int_{G_\alpha} f(y)\,d\mu^{[t(g)]} = \int_{G_\alpha} f(gy) \,d\mu^{[s(g)]}$$
follows as in the proof of \cref{Haarquotient}. Next, we need to prove continuity of $[x] \mapsto \int_{G_\alpha} f(y)\,d\mu^{[x]}$ for fixed $f \in \mathcal{C}_c(G_\alpha)$.  Consider $[x] \in \G0_\alpha$ with representative $x \in \G0$, and choose $N \geq 0$ such that $st(x, \coverU_N) \subset K_{N - 1}$.  Given $\epsilon > 0$ we can use the proof of  \cref{eqn.Haarsystemapproximation.continuity} to find an $n > N$ such that for any $V \in \mathcal{U}_{n}^0$ and $y, z \in V \cap K_{n - 1}$ we have
\begin{equation*}
\left|\int_{G} \tilde{f}\,d\mu^y - \int_{G} \tilde{f} \,d\mu^z\right| < \epsilon.    
\end{equation*}
In the pseudo-metric $d$, the cover by $\frac{1}{2^n}$-balls refines $\mathcal{U}_n$, so there exists a $W \in \mathcal{U}_n$ such that $\{y: d(x, y) < \frac{1}{2^n}\} \subset W$; since $st(x, \coverU_n) \subset st(x, \coverU_N) \subset K_{N - 1}$, we have $W \subset K_{n - 1}$.  It follows that for all $[y] \in \G0_\alpha$ for which $d([x], [y]) < \frac{1}{2^n}$ we have 
\begin{equation*}
\left|\int_{G_\alpha} f\,d\mu^{[x]} - \int_{G_\alpha} f\,d\mu^{[y]}\right| < \epsilon,    
\end{equation*}
proving the continuity condition of the Haar system.  This concludes the proof of the fact that $\{\mu^{[x]} : [x] \in \G0_\alpha\}$ is a Haar system on $G_\alpha$. The fact that the pullback map induces a *-morphism of convolution algebras follows from \cref{inducedmorphism}.
\end{proof}

\begin{Theorem}\label{sexyapproximationtheorem}
If $(G,\sigma)$ is a $\sigma$-compact groupoid with Haar system of measures and 2-cocycle $\sigma:\to \TT$, then $G$ admits an inverse approximation by second countable and locally compact groupoids $\{(G_\alpha,\sigma), p^\beta_\alpha, A\}$ as in \cref{approximate} and, furthermore, each groupoid in the approximation admits a Haar system of measures which makes all the bonding maps and projections mappings onto the inverse system Haar system preserving.
\end{Theorem}
%
%\begin{Theorem}
%If $(G,\sigma)$ is a $\sigma$-compact groupoid with open source and target maps and $\sigma:\to \TT$ is a 2-cocycle on $G$, then $G$ admits an inverse approximation by second countable and locally compact groupoids $\{(G_\alpha,\sigma), p^\beta_\alpha, A\}$ as in \cref{approximate}. If $G$ is \'etale then each groupoid in the approximation can be chosen to be \'etale. Moreover, if $G$ has a Haar system of measures, then each groupoid $G_\alpha$ can be chosen to have a Haar system of measures such that the canonical pullback morphisms $C_c(G_\alpha,\sigma_\alpha) \to C_c(G_\beta,\sigma_\beta)$ are *-inclusions for all $\beta \ge \alpha$ and we have that $C_c(G,\sigma) = \varinjlim_\alpha C_c(G_\alpha,\sigma_\alpha)$ in the category of topological *-algebras over $\CC$.
%\end{Theorem}
%
\begin{proof}
Let $\mathcal{N}$ denote the collection of groupoid normal sequences on $G$.  $\mathcal{N}$ with the order  $\alpha \leq \beta$ if $\alpha$ cofinally refines $\beta$ is directed (i.e. forms a filter under the poset structure) because the collection of groupoid normal sequences are cofinal in the collection of all normal sequences.

If $\gamma \leq \theta$ then the equivalence relation determined by the normal sequence $\gamma$ has strictly smaller equivalence classes than the equivalence relation determined by $\theta$; furthermore, there is a unique function $q^\theta_\gamma :G_\theta \to G_\gamma$ such that $q_\gamma = q^\theta_\gamma \circ q_\theta$. One can also easily check that $q_\theta^\gamma$ is a proper surjective groupoid homomorphism. 

$\mathcal{N}$ induces an inverse system $\{G_\alpha,  q^\alpha_\beta:G_\alpha\to G_\beta, \mathcal{N}\}$ of groupoids.  We claim that $G$ is the inverse limit of this system with projection maps given by $\{q_\alpha:G\to G_\alpha:\alpha \in \mathcal{N}\}$. It is easy to see that the natural map $G\to \varprojlim_\alpha G_\alpha$ given by $g \to (q_\alpha(g))_\alpha$ is a homeomorphism (see \cref{topapprox}) as the groupoid normal sequences are, by \cref{Coverings},  cofinal (in the order induced by cofinal refinement) in the directed set of all normal sequences of covers for $G$. It is clearly a groupoid homomorphism and thus it is topological groupoid isomorphism.

In the case where $G$ has a Haar system of measures and 2-cocycle $\sigma$, we know that $G_\alpha$ can be chosen to have a Haar system of measures and cocycle $\sigma_\alpha$ such that the cannonical pullback morphisms $C_c(G_\alpha,\sigma_\alpha) \to C_c(G,\sigma)$ are *-morphisms. The maps $q^\alpha_\beta:G_\alpha \to G_\beta$ have the property that $(q^\alpha_\beta)^*:C_c(G_\beta,\sigma_\beta) \to C_c(G_\alpha,\sigma_\alpha)$ are *-morphisms for all $\alpha \ge \beta$ (it follows from essentially the same argument as \cref{Haarsystemapproximation}). For purely topological considerations, we have that $\varinjlim_\alpha C_c(G_\alpha,\sigma_\alpha) = C_c(G,\sigma)$ in the category of topological *-algebras over $\CC$.
\end{proof}

\begin{Corollary}
Let $G$ and $\{G_\alpha,p^\alpha_\beta,A\}$ be as in \cref{sexyapproximationtheorem}. The canonical pullback maps induce $I$-norm preserving $*$-embeddings $C_c(G_\alpha,\sigma_\alpha) \to C_c(G_\beta,\sigma_\beta)$ of twisted convolution algebras when $\beta \ge \alpha$ and, furthermore, we have that  $C_c(G,\sigma) = \varinjlim_\alpha C_c(G_\alpha,\sigma_\alpha)$ in the category of topological *-algebras over $\CC$.
\end{Corollary}

\begin{Corollary}
Let $G$ be a $\sigma$-compact groupoid with Haar system of measures and 2-cocycle $\sigma:\G2\to \TT$ and let $\{(G_\alpha,\sigma_\alpha),p^\alpha_\beta:A\}$ be the inverse approximation as in the previous theorem. The maximal groupoid $C^*$-algebra functor takes the directed system $\{C_c(G_\alpha,\sigma_\alpha),(p^\alpha_\beta)^*,A\}$ to a directed system $\{C^*(G_\alpha,\sigma_\alpha),(p^\alpha_\beta)^*,A\}$ and, moreover, $\varinjlim_\alpha C^*(G_\alpha,\sigma_\alpha) = C^*(G,\sigma)$.
\end{Corollary}
\begin{proof}
The first assertion follows from \cref{inducedmorphism}. We just need to prove the second assetion about direct limits.

Recall that a $C^*$-algebra is the direct limit of a directed system $\{A_\alpha,p^\alpha_\beta,D\}$ if there exists a mapping of the inverse system to $A$ that commutes with the system and such that the union of the images of the $C^*$-algebras in the system are dense in $A$. It is evident that the mapping $i^\alpha_\beta:C_c(G_\alpha,\sigma_\alpha) \hookrightarrow C_c(G,\sigma)$ induces a morphism $j^\alpha_\beta:C^*(G_\alpha,\sigma_\alpha) \to C^*(G,\sigma)$ and it is also clear that this pullback morphism commutes with all the other pullback morphisms in the direct system. Notice also that the union of the images clearly contains $C_c(G,\sigma)$ and hence is dense. 
\end{proof}

%
%\begin{Corollary}\label{completion}
%Let $G$ be a $\sigma$-compact groupoid with Haar system of measures and 2-cocycle $\sigma:\G2\to \TT$ and let $\{(G_\alpha,\sigma_\alpha),p^\alpha_\beta:A\}$ be the inverse approximation as in \cref{mainthm.approximation}. Choose a $C^*$-algebra completion $C_\pi^*(G,\sigma)$ of $C_c(G\sigma)$. For each $\alpha\in A$, let $C_{\pi_\alpha}^*(G_\alpha,\sigma_\alpha)$ denote the closure of the twisted convolution algebra $C_c(G_\alpha,\sigma_\alpha)$ in  $C_\pi^*(G,\sigma)$. It follows that  $\{C_{\pi_\alpha}^*(G_\alpha,\sigma_\alpha), (p^\alpha_\beta,A\}$ is a directed system of $C^*$-algebras with direct limit $C_\pi^*(G,\sigma)$.
%\end{Corollary}
%
%
%
Applying the construction of \cref{mainthm.approximation} to the specific examples of groups and topological spaces respectively, we easily obtain the following two corollaries, stated here for future reference:

\begin{Corollary}
If $G$ is a $\sigma$-compact group with  cocycle $\sigma$ then there exists an inverse approximation $\{(G_\alpha,\sigma_\alpha),p^\alpha_\beta,A\}$ of $G$ by second countable groups $G_\alpha$ with cocycle $\sigma_\alpha$ and proper continuous and cocycle preserving bonding homomorphisms $p^\alpha_\beta$  such that  $C^*(G,\sigma)$ (the maximal completion) is an inductive limit of the directed system $\{C^*(G_\alpha,\sigma_\alpha),(p^\alpha_\beta)^*,A\}$.
\end{Corollary}

\begin{Corollary}
If $X$ is a locally compact, Hausdorff, and  $\sigma$-compact space then  there exists an inverse approximation $\{X_\alpha,p^\alpha_\beta,A\}$ of $X$ by locally compact, Hausdorff, and second countable spaces $X_\alpha$ with proper and continuous bonding maps such that $C_0(X)$ is the directed limit of the dual directed system $\{C_0(X_\alpha),(p^\alpha_\beta)^*,A\}$.
\end{Corollary}

\subsection{Properties of the Approximations} \label{section.approximationproperties}
In this section, we discuss some basic properties of groupoids, and whether or not the approximation construction preserves those properties, or can be suitably modified so the property passes to the approximation groupoids.

\subsubsection{Transitivity}
If $G$ is transitive (that is, for every $x, y \in \G0$ there exists $g \in \G1$ such that $s(g) = x$ and $t(g) = y$), it is not at all hard to see from the construction in \cref{quotientconstruction} that each of the quotient groupoids $G_\alpha$ is also transitive: in the quotient groupoid we have $s([g]) = [s(g)]$ and $t([g]) = [t(g)]$ for any $g \in \G1$.

\subsubsection{Principality}
If $G$ is a principal  groupoid (i.e. the stabilizer groups $G_x^x$ are trivial for all $x \in \G0$), it is not necessarily true that $G_\alpha$ will necessesarily also be principal.  For example, consider the groupoid $G$ (and the approximation) shown below:

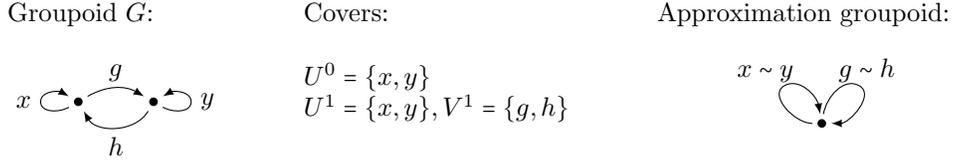
\null\hfill
\begin{minipage}[t]{.2\linewidth}
Groupoid $G$: \\

\begin{tikzpicture}
\node (a) at (0, 0) {};
\node (b) at (1, 0) {};

\filldraw (a) circle(0.05);
\filldraw (b) circle(0.05);

% self-arrows
\draw[->] (a) to[out=-150, in=-90] (-0.5, 0) node[anchor=east]{$x$} to[out=90, in=150] (a);
\draw[->] (b) to[out=-30, in=-90] (1.5, 0)  node[anchor=west]{$y$} to[out=90, in=30] (b);
% arrows between different vertices
\draw[->] (a) to[out=30, in=150] node[midway, label={[yshift=-.15cm]$g$}]{} (b);
\draw[->] (b) to[out = -120, in=-60] node[midway, anchor=north]{$h$} (a);
\end{tikzpicture}
\end{minipage}
\hfill
\begin{minipage}[t]{.25\linewidth}
Covers: \\

$U^0 = \{x, y\}$ \\
$U^1 = \{x, y\}, V^1 = \{g, h\}$
\end{minipage}
\hfill
\begin{minipage}[t]{.28\linewidth}
Approximation groupoid:
\begin{center}
\begin{tikzpicture}
\node (a) at (0, 0) {};

\filldraw (a) circle(0.05);

\draw[->] (a) to[out=180, in=-150] (-0.5, .5) node[label={[yshift=-.2cm, xshift=-.25cm]$x \sim y$}]{} to[out=30, in=100] (a);
\draw[->] (a) to[out=80, in=150] (0.5, .5) node[label={[yshift=-.2cm, xshift=.1cm]$g \sim h$}]{} to[out=-30, in=0] (a);
\end{tikzpicture}
\end{center}
\end{minipage}
\hfill\null

On the arrows of $G$ we place the discrete topology, and we use the groupoid normal sequence of covers $\coverU^0_n = \{U^0\}$, $\coverU^1_n = \{U^1, V^1\}$ to perform the approximation.  Note that $G$ is a principal groupoid, and the normal sequence $\coverU$ satisfies the conditions of \cref{Coverings}; however, the quotient groupoid ($\ZZ_2$, also shown above) is not principal.

An interesting question to address is whether or not the constructions of the covers can be changed so that the approximation groupoids are, in general, principal when $G$ is.
\subsubsection{\'Etalness}
Recall that a topological groupoid $G$ is said to be \textbf{\'etale} if the source map $s$ and the target map $t$ are both local homeomorphisms.  In the case when $G$ is \'etale, we can modify the construction so that the approximation groupoids $G_\alpha$ are also \'etale.

We use the fact that a groupoid $G$ is \'etale if and only if $\G0$ is open in $G$ and $G$ has a Haar system of measures (c.f. page 2 of \cite{Renault}).  Since $\G0$ is clopen in $G$, in constructing any of the approximation groupoids one can modify the cover $\{\mathcal{U}^1_0\}$ of $G$ so that, if $U\in \mathcal{U}_0^1$, then either $U \subset \G0$ or $U \subset (G \setminus \G0)$.  Then, in the pseudo-metric induced by $\{\coverU_n\}$, the object space and arrow space of $G$ have distance 1 from each other, and so $\G0_\alpha$ is clopen in the approximation groupoid $G_\alpha$. Since $G_\alpha$ has an induced Haar system of measures by \cref{sexyapproximationtheorem}, it follows that $G_\alpha$ is \'etale.

\section{Equivalence of Groupoids}\label{renaultequivalencetheorem}
The main result in this section is \cref{equivalence}, which shows that equivalent $\sigma$-compact groupoids have  Morita equivalent $C^*$-algebras.  We will use the terminology from Section 1 of \cite{SW}.  Recall that, for a groupoid $G$, a subset $A\subset \G0$ is said to be \textbf{full} if $G \cdot A = \G0$.  If $G$ is a groupoid and $A\subset \G0$, then we use the notation $G(A)$ to denote the subgroupoid of $G$ whose arrows begin and end in $A$; i.e. the set of all $g\in G$ with $s(g),t(g) \in A$.
\begin{Definition} If $G$ and $H$ are topological groupoids, we say that a groupoid $L$ is a \textbf{linking groupoid for $H$ and $G$} if $L^{(0)} = H^{(0)}\sqcup \G0$ such that $\G0$ and $\gpdH0$ are full clopen subsets of $\gpdL0$ and such that $L(\G0) = G$ and $L(\gpdH0) = H$. 
\end{Definition}
If $G$ and $H$ have Haar systems of measures, then a linking groupoid $L$ for $G$ and $H$ also has a Haar system of measures (see Lemma 4 of \cite{SW}) that restricts to the Haar system on $G$ and $H$ respectively.
\begin{Definition} Suppose $Z$ is a topological space for which there is a continuous and open map $r_Z : Z \to \gpdG0$ (called a \textbf{moment map}).  Define
$$G \ast Z = \{(g, z) \in G \times Z : s(g) = r_Z(z)\}$$
with the relative product topology.  A \textbf{left action} of $G$ on $Z$ is a continuous map $G \ast Z \to Z$ such that $r_Z(z) \cdot z = z$ for all $z \in Z$ and if $(g, h) \in \gpdG2$ and $(h, z) \in G \ast Z$ then $(gh, z) \in G \ast Z$ and $gh \cdot z = g \cdot (h \cdot z)$.

Say that the action is \textbf{free} if $g \cdot z = z$ implies $g = r_Z(z)$.  The action is called \textbf{proper} if the map $G \ast Z \to Z \times Z$ given by $(g, z) \mapsto (g \cdot z, z)$ is proper.

A \textbf{right action} is defined analogously, with the difference that the moment map is denoted $s_Z : Z \to \gpdG0$ and $z \cdot g$ is defined if and only if $s_Z(z) = r(g)$.
\end{Definition}
For a left action of $G$ on $Z$, $\lorbit{G}{Z}$ denotes the orbit space of the action; for a right action, we use the notation $\rorbit{Z}{G}$.
\begin{Definition} $G$ and $H$ are said to be \textbf{equivalent} if there exists a locally compact Hausdorff space $Z$, called a \textbf{$(G,H)$-equivalence}, such that the following conditions hold:
\begin{enumerate}
    \item $Z$ is a free and proper left $G$-space.
    \item $Z$ is a free and proper right $H$-space.
    \item The actions of $G$ and $H$ on $Z$ commute.
    \item The moment map $r_Z : Z \to \gpdG0$ induces a homeomorphism of $\rorbit{Z}{H}$ onto $\G0$.
    \item The moment map $s_Z : Z \to \gpdH0$ induces a homeomorphism of $\lorbit{G}{Z}$ onto $\gpdH0$.
\end{enumerate}
\end{Definition}
Two groupoids $G$ and $H$ are equivalent if and only if there exists a linking groupoid for $G$ and $H$ (Lemma 3 of \cite{SW} shows how to construct a linking groupoid from an equivalence; conversely, if $L$ is a linking groupoid for $G$ and $H$, and we let $A = \gpdG0$, $B = \gpdH0$, then one can show that $Z = L^A_B$ is a $(G, H)$- equivalence).  Note that these results are usually stated under the assumptions that $G$ and $H$ are second countable; however, the proofs do not use the second countability assumption and go through without change for $\sigma$-compact.

The following lemma was suggested to the authors by Dana Williams and its proof is essentially the same as the one given in his book \cite{Williams} (which is currently in draft stage) on groupoid $C^*$-algebras. 

\begin{Lemma}
Let $G$ and $H$ be $\sigma$-compact groupoids with open source and target maps, and let $Z$ be a $(G,H)$-equivalence.  Then $Z$ must itself be $\sigma$-compact.
\end{Lemma}
\begin{proof}
 By assumption, $G$ and $H$ are $\sigma$-compact, and since $\lorbit{G}{Z}$ is homeomorphic to $\gpdH0$, so is $\lorbit{G}{Z}$.  Let $\lorbit{G}{Z} = \cup_n K_n$ be an exhaustion of $\lorbit{G}{Z}$ by $\sigma$-compact neighbourhoods.  Using the fact that the orbit map $\pi : Z \to \lorbit{G}{Z}$ is continuous and open (see Proposition 2.1 of \cite{MW3}), we can lift each $K_n$ to a compact set $T_n$ in $Z$ such that $\pi(T_n) = K_n$.  Namely, at each $z \in Z$ choose a relatively compact open neighbourhood $U_z$.  Then $\{\pi(U_z)\}$ is an open cover of $K_n$, so necessarily it has a finite subcover $\{\pi(U_{z_1}), \ldots \pi(U_{z_k})\}$.  Let $T_n = \mathop{\cup}\limits_{i = 1}^k \left(\overline{U}_{z_i} \cap \pi^{-1}(K_n)\right)$.
 
 Let $G = \cup_n C_n$ be an exhaustion of $G$ by compact neighborhoods.  Since the action of $G$ on $Z$ is continuous, it follows that $\{C_n \cdot T_n\}$ is an exhaustion of $Z$ by compact neighborhoods, concluding the proof that $Z$ is $\sigma$-compact.
\end{proof}

As shown in Lemma 3 of \cite{SW}, the linking groupoid associated to a $(G,H)$-equivalence $Z$ is topologically the disjoint union $G\sqcup H\sqcup Z \sqcup Z^{op}$ with object space $\G0\sqcup \HH$. Being that $G$, $H$ and $Z$ are all $\sigma$-compact, it follows that the linking groupoid associated to an equivalence between $\sigma$-compact groupoids with open source and target maps is $\sigma$-compact. Moreover, every linking groupoid comes from such an equivalence, leading to the following result:
\begin{Corollary}
Let $G$ and $H$ be $\sigma$-compact groupoids with open source and target maps and let $L$ be a linking groupoid for $G$ and $H$. Then $L$ is $\sigma$-compact.
\end{Corollary}
If $L$ is a linking groupoid for $G$ and $H$ we can approximate $L$ in such a way that we get equivalent approximations for $G$ and $H$.

\begin{Proposition}\label{approximatethatlink}
Let $G$ and $H$ be $\sigma$-compact groupoids with Haar systems and let $L$ be a $\sigma$-compact linking groupoid for $G$ and $H$.  We may assume $L$ is endowed with a Haar system that restricts to the given Haar sytems for $G$ and $H$. There exists an approximation $\{L_\alpha,q^\alpha_\beta,A\}$ for $L$ (as in \cref{mainthm.approximation}), with projection maps $q_\alpha : L \to L_\alpha$, such that if we let $H_\alpha = q_\alpha(H)$ and $G_\alpha = q_\alpha(G)$ then:
\begin{enumerate}
\item $\{G_\alpha, \restrict{q^\alpha_\beta}{G_\alpha}, A\}$ is an approximation for $G$ and $\{H_\alpha, \restrict{q^\alpha_\beta}{H_\alpha}, A\}$ is an approximation for $H$, and
\item $L_\alpha$ is a linking groupoid for $G_\alpha$ and $H_\alpha$.
\end{enumerate}  \end{Proposition}
\begin{proof}
We need to tweak the proof of \cref{mainthm.approximation} to require that the covers $\mathcal{U}_n$ in the groupoid normal sequences separate between $G$ and $H$, in the sense that if $U$ is an open set in one of the covers such that $U\cap G \neq \emptyset$ then $U\cap H = \emptyset$, and vice versa; however, it should be clear that this requirement can be enforced, since $G$ and $H$ are clopen sets in $L$. This guarentees that the images of $G$ and $H$ are clopen sets in the approximation $L_\alpha$ determined by the covers $\{\mathcal{U}_n\}$. 

It follows in particular that $L_\alpha^{(0)}$ is a disjoint union of the clopen sets $\G0_\alpha$ and $H^{(0)}_\alpha$. The fact that $\G0_\alpha$ is full in $L_\alpha$, and $L_\alpha(\G0_\alpha) = G_\alpha$ (and similarly for $\gpdH0$) follows immediately from the fact that the quotient map $q_\alpha : L \to L_\alpha$ commutes with both the source and target maps and the action of $L$ on $\gpdG0$.  Hence $L_\alpha$ is a linking groupoid for $G_\alpha$ and $H_\alpha$.
\end{proof}

\begin{Theorem}[Equivalence Theorem] \label{equivalence}
Let $G$ and $H$ be $\sigma$-compact groupoids with Haar systems. If $G$ and $H$ are equivalent then $C^*(G)$ and $C^*(H)$ are  Morita equivalent. 
\end{Theorem}
\begin{proof}
Recall how the equivalence theorem works in the second countable setting (modulo technicalities that we will not need to consider here). Let $L$ be a linking groupoid for the second countable groupoids $G$ and $H$. The idea is to show that $C^*(L)$ is a linking algebra for $C^*(G)$ and $C^*(H)$. Notice first that the characteristic functions $p_G$ and $p_H$ for $\G0$ and $\HH$, respectively, are projections inside the multiplier algebra $M(C^*(L))$. One can show that $pC^*(L)q$ is an imprimitivity $C^*(G)\text{-}C^*(H)$ bimodule.  

In the $\sigma$-compact case, we have the exact same framework. It is immediate that $pC^*(L)q$ is a $C^*(G)\text{-}C^*(H)$ Hilbert bimodule and that it satisfies all the requisite properties for being an $C^*(G)\text{-}C^*(H)$ imprimitivity Hilbert bimodule except for the fullness condition ( definition 3.1 in \cite{RW}). 

By \cref{approximatethatlink} and \cref{mainthm.approximation}, we know that 

\begin{itemize}
    \item $C^*(L)$ is the direct limit of the subalgebras $C^*(L_\alpha)$ (and analogously for $G$ and $H$ using the same approximation)
    \item $L_\alpha$ is a linking groupoid for $G_\alpha$ and $H_\alpha$ for each $\alpha$.
\end{itemize}

It is clear that the projections $p$ and $q$ are equal to the pullbacks of the projections $p_\alpha$ and $q_\alpha$ (coming from the characteristic functions on $\G0_\alpha$ and $H_{\alpha}^{(0)}$) in $M(C^*(L_\alpha))$ (under the induced mapping $M(C^*(L_\alpha)) \to M(C^*(L))$). 

We know that the embedded copy of  $p_\alpha C^*(L_\alpha) q_\alpha$ in $pC^*(L)q$ (with the inherited bimodule structure) is an imprimitivity bimodule for $C^*(G_\alpha)$ and $C^*(H_\alpha)$. It follows that the closure of $span\{\langle x,y\rangle_C^*(G): x,y \in pC^*(L)q\})$ contains $C^*(G_\alpha)$ for all $\alpha$. As $\bigcup_\alpha C^*(G_{\alpha})$ is dense in $C^*(G)$, it follows that $span\{\langle x,y\rangle_C^*(G): x,y \in pC^*(L)q\})$ is dense in $C^*(G)$. The same argument holds for $C^*(H)$ and so it follows that $pC^*(L)q$  is an imprimitivity bimodule for $C^*(G)$ and $C^*(H)$. 
\end{proof}

%
%
%
%
%
%
% Show current TODO lists
%\ifdraftcopy
%\section{To Do Lists}
%\listoftodoeasy
%\listoftodomed
%\listoftodohard
%\else{} %do not show todolists in final version
%\fi

\end{document}